\documentclass[12pt]{amsart}
\usepackage{amsthm,amsmath,amssymb,amsopn}
\usepackage{graphicx,epstopdf,epsfig,subfig}
\usepackage[colorlinks,linktocpage,linkcolor=black]{hyperref}
\usepackage[table,xcdraw]{xcolor}
\usepackage{url,color}
\usepackage{booktabs}
\usepackage{colortbl,multirow}
\usepackage{appendix} 
\numberwithin{equation}{section}
\numberwithin{table}{section}
\numberwithin{figure}{section}

\newtheorem{lemma}{Lemma}[section]
\newtheorem{theorem}{Theorem}[section]

\theoremstyle{definition}

\newcommand{\gv}[1]{\ensuremath{\boldsymbol{#1}}}
\newcommand{\grad}[1]{\gv{\nabla} #1} 
\renewcommand{\div}[1]{\nabla \cdot #1} 

\newcommand{\wt}[1]{\widetilde{#1}}
\newcommand{\wh}[1]{\widehat{#1}}

\newcommand{\MTh}{\mathcal{T}_h}
\newcommand{\supp}{\mathop{\mathrm{supp}}}
\newcommand{\lj}{[ \hspace{-2pt} [}
\newcommand{\rj}{] \hspace{-2pt} ]}
\def\na{\nabla}
\def\pa{\partial}

\def\Lam{\Lambda}
\def\Om{\Omega}
\newcommand{\mb}[1]{\mathbb{#1}}
\newcommand{\mc}[1]{\mathcal{#1}}
\newcommand{\abs}[1]{\left\lvert#1\right\rvert}
\newcommand{\nm}[2]{\|\,#1\,\|_{#2}}
\newcommand{\snm}[2]{\abs{\,#1\,}_{#2}}
\newcommand{\Lr}[1]{\left(#1\right)}

\newcommand{\set}[2]{\{\,#1\,\mid\,#2\}}

\DeclareMathOperator{\argmin}{arg min}

\definecolor{lightgray}{gray}{0.9}

\begin{document}
\title[Stokes Equation]{A Discontinuous Galerkin Method for the Stokes
 Equation by Divergence-free Patch Reconstruction}

\author[R. Li]{Ruo Li}
\address{CAPT, LMAM and School of Mathematical Sciences, Peking
University, Beijing 100871, P.R. China} \email{rli@math.pku.edu.cn}

\author[Z.-Y. Sun]{Zhiyuan Sun}
\address{School of Mathematical Sciences, Peking University, Beijing
  100871, P.R. China} \email{zysun@math.pku.edu.cn}

\author[Z.-J. Yang]{Zhijian Yang}
\address{School of Mathematics and Statistics, Wuhan University, Wuhan
  430072, P.R. China} \email{zjyang.math@whu.edu.cn}

\maketitle
\begin{abstract}
  A discontinuous Galerkin method by patch reconstruction is proposed
  for Stokes flows. A locally divergence-free reconstruction space is
  employed as the approximation space, and the interior penalty method
  is adopted which imposes the normal component penalty terms to
  cancel out the pressure term. Consequently, the Stokes equation can
  be solved as an elliptic system instead of a saddle-point problem
  due to such weak form. The number of degree of freedoms of our
  method is the same as the number of elements in the mesh for
  different order of accuracy. The error estimations of the proposed
  method are given in a classical style, which are then verified by
  some numerical examples.



\end{abstract}
\section{Introduction}
The incompressible Stokes equations describe the low Reynolds number
flows which are the linearization of the Navier-Stokes equations
\cite{temam2001navier}. In the finite element method for
incompressible Stokes problem see, e.g. \cite{crouzeix1973conforming},
the major concern is how to impose the incompressibility
condition. The conforming mixed finite element methods are usually
introduced weakly by the Lagrange multiplier, and solve a saddle
points problem. The methods are required to setup the approximation
spaces for velocity and pressure to satisfy the in-sup conditions,
which is necessary to guarantee the numerical stability. Consequently,
the numerical solution obtained is weakly divergence-free, we refer
readers to \cite{boffi2013mixed} for more details.

The other approach to solve the incompressible Stokes equation is
constructing a divergence-free space to automatically fulfill the
incompressible condition and eliminate the pressure. There are many
efforts on constructing global divergence-free basis function. For
example, the Crouzeix-Raviart elements was proposed in
\cite{crouzeix1973conforming}, which are divergence-free in each
element and continuous at the midpoint of element edges. The $H(div)$
finite elements was proposed in \cite{wang2009robust}. The methods to
construct these divergence-free spaces are quite subtle that it is a
nontrivial job to extend these methods to a grid with its elements in
unusual geometry. An alternative way is abandoning the normal
continuity of basis function, and using the locally divergence-free
elements which are proposed in \cite{baker1990piecewise}. See the
recent developments in this fold in
\cite{karakashian1998nonconforming, cockburn2007note, liu2011penalty}.

The discontinuous Galerkin (DG) method by patch reconstruction was
recently introduced in \cite{li2016discontinuous} to solving the
elliptic equation, then was developed in \cite{li2017discontinuous}
and \cite{li2018finite} to various model problems. We are motivated to
approximate the divergence-free velocity in Stokes problem using the
discontinuous space by patch reconstruction. Actually, we propose a
locally divergence-free space by imposing the constrain locally in
patch reconstruction. With the new space, we adopt the interior
penalty method to Stokes problems \cite{hansbo2008piecewise}. The
penalty term is introduced on the normal component to control the
inconsistent error. Consequently, an elliptic problem is attained to
be solved instead of a saddle-point problem. To construct the locally
divergence-free space, we use the piecewisely solenoidal polynomial
functions only. The space has one degree of freedom in each
element. Therefore, the dimension of approximation space does not
depend on the polynomial orders. The interpolation error estimate of
this space can be obtained naturally, which leads to the approximate
error estimate of the Stokes problem following the standard
techniques. We note that the new space is a subspace of canonical
locally divergence-free DG space. Our method enjoys the advantages of
DG method, for example that it works on meshes with polygonal
elements. We demonstrate by numerical examples on different meshes.

The rest of this paper is organized as follows. In Section
\ref{sec:basis}, we describe the method to construct the locally
divergence-free approximation space and and give the approximation
estimate of this space. Then we present the interior penalty method of
discontinued Galerkin method for Stokes problems in Section
\ref{sec:weakform}. The error estimate of the proposed method is
obtained under the DG energy norm. Finally, numerical results for
two-dimensional examples are presented in Section \ref{sec:examples}
to validate our estimates and demonstrate the capacity of our method.


\section{Construction of Approximation Space}\label{sec:basis}
Let $\Om$ be a polygonal domain in $\mb{R}^2$. The mesh $\MTh$ is a
partition of $\Om$ with polygons denoted by $K$ that $\bigcup_{K \in
  \mathcal{T}_h} \bar{K} = \Omega$. For any two elements $K_0$, $K_1
\in \mathcal{T}_h$, $K_0 \bigcap K_1 = \emptyset$ if $K_0 \neq
K_1$. Here $h{:}=\max_{K\in\MTh}h_K$ with $h_K$ the diameter of
$K$. We denote by $\abs{K}$ the area of $K$. $\Gamma_h$ denotes the union
of boundaries of element $K\in \MTh$, where $\Gamma_h^0$ is the union
of interior boundaries $\Gamma_h^0 \triangleq \Gamma_h \setminus
\partial \Omega$. We assume that the partition $\MTh$ possesses the
following shape regularity conditions ~\cite{Brezzi:2009,DaVeiga2014}:

\begin{enumerate}
\item[{\bf A1}\;]There exists an integer $N$ independent of $h$, that
  any element $K$ admits a sub-decomposition $\wt{\MTh}|_K$ which
  consists of at most $N$ triangles $T$.
\item[{\bf A2}\;] $\wt{\MTh}$ is a compatible sub-decomposition, that
  any $T\in\wt{\MTh}$ is shape-regular in the sense of
  Ciarlet-Raviart~\cite{ciarlet:1978}: there exists a real positive
  number $\sigma$ independent of $h$ such that $h_T/\rho_T\le\sigma$,
  where $\rho_T$ is the radius of the greatest ball inscribed in $T$.
\end{enumerate}

Let $D$ to be a subdomain of $\Om$, which may be an element or an
aggregation of the elements belong to $\MTh$. Let $H^{m}(D)$ demote
the Sobolev space of real valued functions on $D$ with a positive
integer $m$. For vector valued functions, we introduce the spaces
$\gv{H}^{m}(D)=[H^{m}(D)]^2$. Let $\mb{P}_m(D)$ be a set of polynomial
with total degree not greater than $m$ on domain $D$, and the
corresponding vector valued polynomial spaces $[\mb{P}_m(D)]^2$ is
denoted by $\pmb{\mb{P}}_m(D)$.

Assumptions {\bf A1} and {\bf A2} allow quite general shapes in the
partition, such as non-convex or degenerate elements. They also lead
to some common used properties and inequalities:

\begin{enumerate}
\item[{\bf M1}]$\forall T\in\wt{\MTh}$, there exists $\rho_1\ge 1$
  that depends on $N$ and $\sigma$ such that $h_K/h_T\le\rho_1$.

\item[{\bf M2}][{\it Agmon inequality}]\;There exists a constant $C$
  that depends on $N$ and $\sigma$, but independent of $h_K$ such that
\begin{equation}\label{eq:agmon}
\nm{v}{L^2(\pa K)}^2\le C\Lr{h_K^{-1}\nm{v}{L^2(K)}^2+h_K\nm{\na
    v}{L^2(K)}^2},\qquad\text{for all\quad}v\in H^1(K).
\end{equation}

\item[{\bf M3}][{\it Approximation property}]\;There exists a constant
  $C$ that depends on $N,r$ and $\sigma$, but independent of $h_K$
  such that for any $v\in H^{r+1}(K)$, there exists an approximation
  polynomial $\wt{v}\in\mb{P}_r(K)$ such that
\begin{equation}\label{eq:app}
\nm{v-\wt{v}}{L^2(K)}+h_K\nm{\na(v-\wt{v})}{L^2(K)}\le
Ch_K^{r+1}\snm{v}{H^{r+1}(K)}.
\end{equation}

\item[{\bf M4}][{\it Inverse inequality}]\;For any $v\in\mb{P}_r(K)$,
  there exists a constant $C$ that depends only on $N,r,\sigma$ and
  $\rho_1$ such that
\begin{equation}\label{eq:inverse}
\nm{\na v}{L^2(K)}\le Ch_K^{-1}\nm{v}{L^2(K)}.
\end{equation}

\item[{\bf M5}][{\it Discrete trace inequality}]\;For any
  $v\in\mb{P}_r(K)$, there exists a constant $C$ that depends only on
  $N,r,\sigma$ and $\rho_1$ such that
\begin{equation}\label{eq:traceinv}
\nm{v}{L^2(\pa K)}\le C h_K^{-1/2}\nm{v}{L^2(K)}.
\end{equation}

\end{enumerate}
The above four
inequalities~\eqref{eq:agmon},~\eqref{eq:app},~\eqref{eq:inverse}
and~\eqref{eq:traceinv} are hold for vector valued space
$\gv{H}^{m}(D)$ and $\pmb{\mb{P}}_m(D)$ with corresponding norms,
\begin{equation*}
 \|\gv v\|_{H^{m}(D)}^2\triangleq\sum_{i=1}^2\|\gv v_i\|_{
   H^{m}(D)}^2,\quad \gv v\in \gv{H}^m(D).
\end{equation*}
We are interested in the spaces of solenoidal vector
fields
\[
\gv{S}^m(D)=\{\gv{v}\in \gv{H}^m(D):\div \gv{v}=0 \quad \text{in} \ D
\},
\]
and the polynomial solenoidal vector field
\[
\mb{S}_m(D)=\{\gv{v}\in \pmb{\mb{P}}_m(D):\div \gv{v}=0 \quad
\text{in} \ D \}.
\]
The above inequalities are also hold for the $\gv{S}^m(D)$ and
$\mb{S}_m(D)$, see in~\cite{baker1990piecewise}. Here we restate the
approximation property:

{\bf M3} [{\it Approximation property}]\;There exists a constant $C$
that depends on $N,r$ and $\sigma$, but independent of $h_K$ such that
for any $\gv{v}\in \gv{S}^{r+1}(K)$, there exists an approximation
polynomial $\wt{\gv v}\in\mb{S}_r(K)$ such that
\begin{equation}\label{eq:app-restate}
\nm{\gv v-\wt{\gv v}}{L^2(K)}+h_K\nm{\na(\gv v-\wt{\gv v})}{L^2(K)}\le
Ch_K^{r+1}\snm{\gv v}{H^{r+1}(K)}.
\end{equation}

In~\cite{li2016discontinuous}, we introduced the reconstruction
operator which mapping the piecewise constant space to discontinuous
piecewise polynomial space. In this paper, we intend to construct a
reconstruction operator $\mc{S}$ that embeds the piecewise constant
vector valued space to discontinuous piecewise solenoidal vector
fields.  For each element $K\in\MTh$, we assign a sampling node or
collocation point $\gv{x}_K \in K$ and element patch $S(K)$. $S(K)$
usually contains $K$ and some elements around $K$. It is quite
flexible to assign the sampling nodes and construct the element
patch. Usually, we let the barycenter of the element $K$ to be the
sampling node, while a perturbation is allowed,
cf. ~\cite{li2016discontinuous}. The element patch is built up by
adding Von Neumann neighbors (adjacent edge-neighboring elements)
recursively until the size of the element patch reaches a fixed
number, for the details we refer to~\cite{li2016discontinuous,
  li2012efficient}.

Let $\mc{I}(K)$ denote the set of the sampling nodes belonging to
$S(K)$ with $\#\mc{I}(K)$ denote its cardinality
\[
\mc{I}(K) = \left\{ \gv{x}_{K'} | K' \in S(K) \right\},
\]
and let $\#S(K)$ be the number of elements belonging to $S(K)$. These
two numbers $\#\mc{I}(K)$ and $\#S(K)$ are equal to each other. We
define $d_K{:}=\text{diam}\;S(K)$ and $d=\max_{K\in\MTh}d_K$.

Denote the piecewise constant space associated with $\MTh$ by $U_h$,
i.e.,
\[
U_h{:}=\set{v\in L^2(\Om)}{v|_K\in\mb{P}_0(K)}.
\]
and let $\gv{U}_h=[U_h]^2$ to be the piecewise constant vector valued
space.

For any $\gv{v}\in \gv{U}_h$ and for any $K\in\MTh$, we reconstruct a
high order solenoidal polynomials $\mc{S}_K \gv{v}$ of degree $m$ by
solving the following discrete least-square problem.
\begin{equation}\label{eq:leastsquares}
\mc{S}_K
\gv{v}=\arg \min_{\gv{p}\in\mb{S}_m(S(K))}\sum_{\gv{x}\in\mc{I}(K)}\abs{\gv{v}(\gv{x})-\gv{p}(\gv{x})}^2.
\end{equation}

Though $\mc{S}_K \gv{v}$ gives a solenoidal approximation on the
entire element patch $S(K)$, but we only use the approximation on the
element $K$. Then the reconstruction operator $\mc{S}_K$ can extended
to the function space $[C^0(S(K))]^2\cap \gv{S}^m(S(K))$, still denote
by $\mc{S}_K$ without ambiguity,
\[
 \mc{S}_K: ~\gv{v} \mapsto \mc{S}_K \gv{v} = \mc S_K\tilde{\gv{v}},
 \quad \forall \gv{v}\in [C^0(S(K))]^2\cap \gv{S}^m(S(K)),
\]
where $\tilde{\gv{v}}( \gv{x}_K')=\gv{v}( \gv{x}_K'), \forall
\gv{x}_K' \in \mc{I}(K)$.

We make the following assumption on the sampling node
set $\mc{I}(K)$.
\noindent\vskip .5cm {\bf Assumption A}\; For any $K\in\MTh$ and
$\gv{p}\in\mb{S}_m(S(K))$,
\begin{equation}\label{assumption:uniqueness}
\gv{p}|_{\mc{I}(K)}=0\quad\text{implies\quad}\gv{p}|_{S(K)}\equiv 0.
\end{equation}

This assumption can guarantee the uniqueness of the solution of discrete
least-square problem~\eqref{eq:leastsquares}. Obvious, a necessary
condition for the solvability of ~\eqref{eq:leastsquares} is that
$\#\mc{I}(K)$ is greater than $(m+1)(m+4)/4$. {\bf Assumption A} is
equal to the following quantitative estimate
\[
\Lambda(m, \mc{I}(K))<\infty
\]
with
\begin{equation}\label{eq:cons}
\Lambda(m, \mc{I}(K)){:}=\max_{\gv{p}\in\mb{S}_m(S(K))}
\dfrac{\nm{\gv{p}}{L^\infty(S(K))}}{\nm{\gv{p}|_{\mc{I}(K)}}{\ell_\infty}}.
\end{equation}

If the mesh is quasi-uniform triangulation and each element patch is
convex, the quantitative estimate of the uniform upper bound of
$\Lambda(m, \mathcal I_K)$ for the general polynomial is obtained in
\cite{li2012efficient}. The requirements of the uniform upper bound
are hard to be satisfied while the polygonal meshes are applied. We
refer the readers to \cite{li2016discontinuous} which gives the
uniform upper bound under milder assumptions, such as polygonal
partition and non-convex element patch. Due to the solenoidal
polynomial space is the subspace of the general polynomial vector
valued space, $\Lambda(m, \mathcal I(K))$ in \eqref{eq:cons} is
uniformly bounded under the same assumption.

\begin{lemma}\label{theorem:localapp}
  If {\em Assumption A} holds, then there exists a unique solution of
  \eqref{eq:leastsquares}, denoted by $\mc{S}_K \gv{v}$, for any
  $K\in\MTh$. Moreover $\mc{S}_K$ satisfies
\begin{equation}\label{eq:invariance}
\mc{S}_K \gv{g}=\gv{g} \quad\text{for
  all\quad}\gv{g}\in\mb{S}_m(S(K)).
\end{equation}
The stability property holds true for any $K\in\MTh$ and $\gv{g}\in
[C^0(S(K))]^2$ and $\div \gv{g}=0$ as
\begin{equation}\label{eq:continuous}
\nm{\mc{S}_K \gv{g}}{L^{\infty}(K)}\le\Lambda(m , \mc{I}(K)) \sqrt{\#
  \mc{I}(K)}\nm{\gv{g}|_{\mc{I}(K)}}{\ell_\infty},
\end{equation}
and the quasi-optimal approximation property is valid in the sense
that
\begin{equation}\label{eq:approximation}
\nm{\gv{g} -\mc{S}_K \gv{g}}{L^{\infty}(K)}\le\Lambda_m
\inf_{\gv{p}\in\mb{S}_m(S(K))} \nm{\gv{g} - \gv{p}}{L^{\infty}(S(K))},
\end{equation}
where $\Lambda_m{:}=\max_{K\in \MTh} \{1+\Lambda(m,\mc{I}(K))\sqrt{\#
  \mc{I}(K)}\}$.
\end{lemma}
\begin{proof}
  The reconstruction operator $\mc{S}_K$ is a projection operator from
  $[C^0(\Omega)]^2$ to $\mb{S}_m(S(K))$ with the discrete $l_2$ norm,
  the identity~\eqref{eq:invariance} is obvious.

By the {\em Assumption A} and definition of $\Lambda(m,\mc{I}(k))$ in
equation~\eqref{eq:cons},
  \begin{equation}\label{continuous_1}
    \|\mc{S}_K\gv{g}\|_{L^{\infty}(K)} \leq
    \|\mc{S}_K\gv{g}\|_{L^{\infty}(S(K))} \leq
    \Lambda(m,\mc{I}(K))\max_{x\in \mc{I}_k}|\mc{S}_K\gv{g}(x)|.
  \end{equation}
From the projection property of operator $\mc{S}_K$, we have

\begin{equation}\label{continuous_2}
\sum_{x\in \mc{I}(K)} \Lr{\mc{S}_K\gv{g}(x)}^2 \leq \sum_{x\in
  \mc{I}}\Lr{\gv{g}(x)}^2 \leq \#\mc{I}(K) \max_{x\in
  \mc{I}_k}|\gv{g}(x)|^2 .
\end{equation}
Combining ~\eqref{continuous_1} and~\eqref{continuous_2}, we get

\[
\|\mc{S}_K\gv{g}\|_{L^{\infty}(K)} \leq \Lambda(m,\mc{I}(K))
\sqrt{\#\mc{I}(K)} \nm{\gv{g}|_{\mc{I}(K)}}{\ell_\infty}.
\]

Assume that $\gv{p}_0$ is the best approximation of $\gv{g}$ under
$L^{\infty}$ norm, $\gv{p}_0 \in \mb{S}_m(S(K))$, and
\[
\|\gv{g}-\gv{p}_0\|_{L^{\infty}(S(K))} = \inf_{\gv{p}\in
  \mb{S}_m(S(K))} \|\gv{g}-\gv{p}\|_{L^{\infty}(S(K))}.
\]
Then, we have following estimate:
\begin{eqnarray*}
  \|\mc{S}_K\gv{g} -\gv{p}_0\|_{L^{\infty}(K)} &=& \|\mc{S}_K(\gv{g}
  -\gv{p}_0)\|_{L^{\infty}(K)} \\ &\leq& \Lambda(m,\mc{I}(K))
  \sqrt{\#\mc{I}(K)} \max_{x \in \mathcal{I}(K)} |(\gv{g}
  -\gv{p}_0)(x)| \\ &\leq& \Lambda(m,\mc{I}(K)) \sqrt{\#\mc{I}(K)}
  \|\gv{g}-\gv{p}_0\|_{L^{\infty}(S(K))} \\ &=& \Lambda(m,\mc{I}(K))
  \sqrt{\#\mc{I}(K)} \inf_{\gv{p}\in \mb{S}_m(S(K))}
  \|\gv{g}-\gv{p}\|_{L^{\infty}(S(K))}.
\end{eqnarray*}

By triangle inequality, the left side of \eqref{eq:approximation} can
be written as
\begin{eqnarray*}
  \|\gv{g} - \mc{S}_K\gv{g}\|_{L^{\infty}(K)} &\leq& \|
  \gv{g}-\gv{p}_0\|_{L^{\infty}(K)} + \|\mc{S}_K\gv{g}
  -\gv{p}_0\|_{L^{\infty}(K)} \\ &\leq& (1+\Lambda(m,\mc{I}(K))
  \sqrt{\#\mc{I}(K)}) \inf_{\gv{p}\in
    \mb{S}_m(S(K))}\|\gv{g}-\gv{p}\|_{L^{\infty}(S(K))}.
\end{eqnarray*}
Together with the definition of $\Lambda_m$, it implies the
quasi-optimality \eqref{eq:approximation}.
\end{proof}

\begin{lemma}
If {\em Assumption A} holds and $\gv{g}\in [C^0(S(K))]^2\cap
\gv{S}^{m+1}(S(K))$, then there exists a constant $C$ that depends on
$N,\sigma$, $\gamma$ and $\rho_1$ such that
\begin{align}
\nm{\gv{g}-\mc{S}_{K} \gv{g}}{L^{2}(K)}&\le C\Lam_m
h_Kd_K^m\snm{\gv{g}}{H^{m+1}(S(K))}.\label{eq:l2app}\\ \nm{\na(\gv{g}-\mc{S}_K
  \gv{g})}{L^2(K)}&\le C
\Lr{h_K^m+\Lam_{m}d_K^m}\snm{\gv{g}}{H^{m+1}(S(K))}.\label{eq:h1app}\\ \nm{\gv{g}-\mc{S}_{K}
  \gv{g}}{L^{2}(\partial K)}&\le C\Lam_m
h_K^{1/2}d_K^m\snm{\gv{g}}{H^{m+1}(S(K))}.\label{eq:traceapp}
\end{align}
\end{lemma}

\begin{proof}
By~\cite[Theorem 4.3]{baker1990piecewise}, we take
$\gv{p}=\gv{\chi}\in\mb{S}_m$in the right-hand side
of~\eqref{eq:approximation}, where $\gv{\chi}$ is the approximation
solenoidal polynomial of order $m$.  Then
\begin{equation}\label{eq:starapp}
\inf_{p \in\mb{S}_m(S(K))}\nm{\gv g-\gv p}{L^{\infty}(S(K))}\le\nm{\gv
  g- \gv \chi}{L^{\infty}(S(K))}\le Cd_K^m\snm{\gv g}{H^{m+1}(S(K))},
\end{equation}
where $C$ depends on $N,m,\sigma$ and $\gamma$.

Substituting the above estimate~\eqref{eq:starapp}
into~\eqref{eq:approximation}, we obtain
\[
\nm{\gv g-\mc{S}_{K} \gv g}{L^{2}(K)}\le\abs{K}^{1/2}\|\gv
g-\mc{S}_{K} \gv g\|_{L^{\infty}(K)} \le C\Lam_m h_Kd_K^m\snm{\gv
  g}{H^{m+1}(S(K))}.
\]
which gives~\eqref{eq:l2app}.

Then, assume that $\wh{\gv g}_m$ be the approximation polynomial
in~\eqref{eq:app} for function $\gv g$, by the {\em inverse
  inequality}~\eqref{eq:inverse} and the approximation
estimate~\eqref{eq:l2app}, then we have
\begin{align*}
\nm{\na(\gv g-\mc{S}_K \gv g)}{L^2(K)}&\le\nm{\na(\gv g-\wh{\gv
    g}_m)}{L^2(K)} +\nm{\na(\wh{\gv g}_m-\mc{S}_K \gv
  g)}{L^2(K)}\\ &\le Ch_K^m\snm{\gv
  g}{H^{m+1}(K)}+Ch_K^{-1}\nm{\wh{\gv g}_m-\mc{S}_{K}
  g}{L^{2}(K)}\\ &\le Ch_K^m\snm{\gv g}{H^{m+1}(K)}+Ch_K^{-1}\nm{\gv
  g-\wh{\gv g}_m}{L^2(K)} +Ch_K^{-1}\nm{\gv g-\mc{S}_{K} \gv
  g}{L^{2}(K)}\\ &\le C\Lr{h_K^m+\Lam_{m}d_K^m}\snm{\gv
  g}{H^{m+1}(S(K))}.
\end{align*}
This gives~\eqref{eq:h1app}.

Combing the {\em Agmon inequality}~\eqref{eq:agmon}, ~\eqref{eq:l2app}
and ~\eqref{eq:h1app}, one has
\begin{align*}
\nm{\gv g-\mc{S}_K \gv g}{L^2(\partial K)}^2&\le C \Lr{h_K^{-1}\nm{\gv
    g-\mc{S}_K \gv g}{L^2(K)}^2 +h_K\nm{\na(\gv g_m-\mc{S}_K \gv
    g)}{L^2(K)}^2}\\ &\le Ch_K \Lam_{m}^2d_K^{2m}\snm{\gv
  g}{H^{m+1}(S(K))}^2.
\end{align*}
Taking the square root of both sides gives~\eqref{eq:traceapp}
completes the proof.
\end{proof}

A global reconstruction operator $\mc{S}$ is defined by
$\mc{S}|_K=\mc{S}_K$. Given $\mc{S}$, we embed $\gv{U}_h$ into a
piecewise discontinuous solenoidal polynomial finite element space
with its order to be $m$. The approximation space $\gv{V}_h$ is defined by
\[\gv{V}_h=\mc{S}\gv{U}_h.\]

Furthermore, the reconstruction operator $\mc{S}$ can be extended to
function space $[C^0(\Omega)]^2\cap \gv{S}^m(\Omega)$, and we still
denote by $\mc{S}$ without ambiguity,
\[
 \mc{S}: ~\gv{u} \mapsto \mc{S} \gv{u} = \mc S\tilde{\gv{u}}, \quad
 \forall \gv{u}\in [C^0(\Omega)]^2\cap \gv{S}^m(\Omega),
\]
where $\tilde{\gv{u}}\in \gv{U}_h$ and $\tilde{\gv{u}}(
\gv{x}_K)=\gv{u}( \gv{x}_K)$.

The basis function of $\gv{V}_h$ are given by the following
process. Define $e_K\in U_h$ to be the characteristic function
corresponding to $K$,
\begin{displaymath}
  e_{K}(x)=\begin{cases} 1,\ x \in K,\\ 0,\ x \notin K.\\
  \end{cases}
\end{displaymath}

Let $\gv{\lambda}_K$ denote the basis function and it is defined by
the reconstruction process:
\begin{displaymath}
  \gv{\lambda}_K=\begin{cases}
  \mc{S}[e_K,0],\ \text{x-component},\\ \mc{S}[0,e_K],\ \text{y-component}.\\
  \end{cases}
\end{displaymath}
The reconstruction operator can be wrote explicitly with the given
basis functions $\{\gv{\lambda}_K|\forall K\in \MTh\}$,
\begin{equation}\label{interpolation}
  \mathcal{S} \gv{u} = \sum_{K \in \MTh} \gv{u}(\gv{x}_{K}) *
  \gv{\lambda}_K, \quad \forall \gv{u} \in [C^0(\Omega)]^2\cap
  \gv{S}^m(\Omega) .
\end{equation}

Next, we will show the implementation of the proposed method by
an example on square domain $[0,1]\times[0,1]$. Here a third order
solenoidal field reconstruction is considered, the basis functions of
the corresponding solenoidal field $\gv{\xi}_j$ are listed as follows,
\[
\left(
     \begin{array}{c}
            1 \\ 0\\
     \end{array}
   \right), \left(
     \begin{array}{c}
             0\\ 1 \\
     \end{array}
\right), \left(
     \begin{array}{c}
             0\\ x \\
     \end{array}
\right), \left(
     \begin{array}{c}
            x \\ -y \\
     \end{array}
\right), \left(
     \begin{array}{c}
             y \\ 0 \\
     \end{array}
\right),
\]
\[
\left(
     \begin{array}{c}
            0 \\ x^2\\
     \end{array}
   \right), \left(
     \begin{array}{c}
             2 xy\\ -y^2 \\
     \end{array}
\right), \left(
     \begin{array}{c}
             x^2\\ -2xy \\
     \end{array}
\right), \left(
     \begin{array}{c}
            y^2 \\ 0\\
     \end{array}
\right),
\]
\[
\left(
     \begin{array}{c}
            0 \\ x^3\\
     \end{array}
   \right), \left(
     \begin{array}{c}
             3 xy^2\\ -y^3 \\
     \end{array}
\right), \left(
     \begin{array}{c}
             x^2y\\ -xy^2 \\
     \end{array}
\right), \left(
     \begin{array}{c}
             x^3\\ -3x^2y \\
     \end{array}
\right), \left(
     \begin{array}{c}
            y^3 \\ 0\\
     \end{array}
\right).
\]
The total degree of freedoms of the third order solenoidal field is
14, it means the size of element patch at least need to be taken as 7
to guarantee the solvability of the least square problem. Thus we take
$\# S(K)$ as 10 and the barycenter as the sampling node. Figure
\ref{tri_mesh_patch} shows domain and the corresponding triangulation,
\begin{figure}
  \begin{center}
    \includegraphics[width=0.4\textwidth]{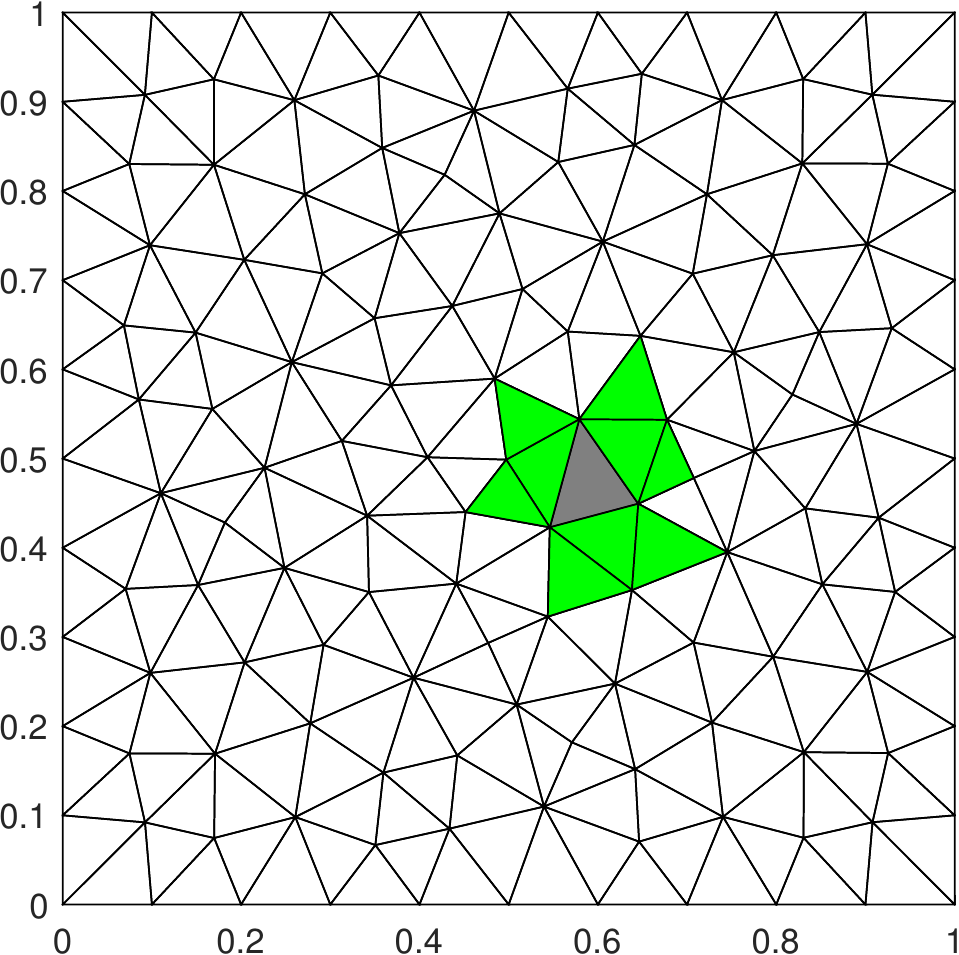}
    \includegraphics[width=0.4\textwidth]{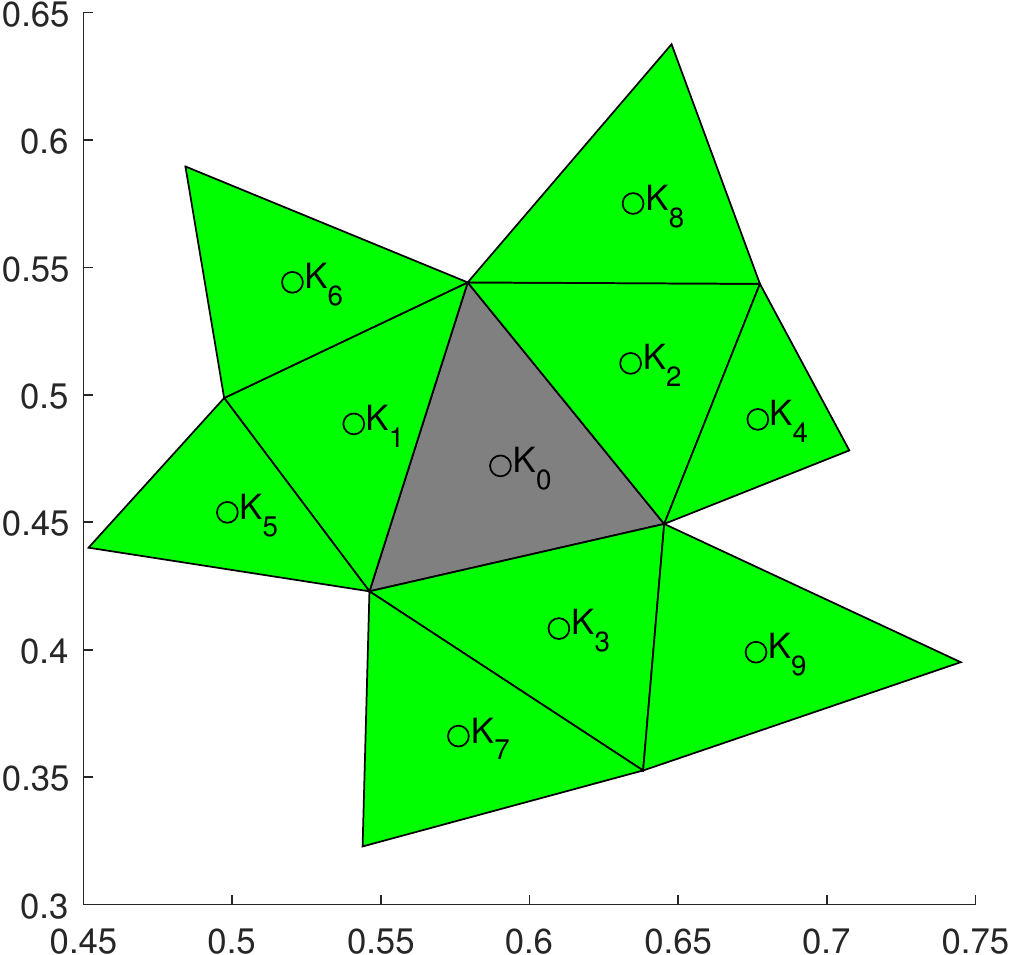}
    \caption{The triangulation (left)/element patch and the sampling
      nodes (right).}
  \label{tri_mesh_patch}
  \end{center}
\end{figure}

Here we take an element $K_0$ as a demonstration element, the element
patch of $K_0$ is taken as
\begin{displaymath}
  S(K_0) = \left\{ K_{0},\cdots, K_{i},\cdots ,K_{9} \right\},\quad
  i=1,2,\cdots,8.
\end{displaymath}
and the corresponding set of the sampling nodes is
\[
\mc{I}(K_0) = \left\{ (x_{K_0},y_{K_0})
,\cdots,(x_{K_i},y_{K_i}),\cdots ,(x_{K_9},y_{K_9})\right\},\quad
i=1,2,\cdots,8.
\]
The element patch and sampling nodes are shown in Figure
\ref{tri_mesh_patch}.

For a given function $\gv{g}=[g_1,g_2]^T\in \gv{S}^m(\Omega)$, the
least square problem \eqref{eq:leastsquares} is as
\begin{equation*}
  \mc S_{K_0} \gv{g}= \argmin_{\{s_j\}\in\mb{R}}
  \ \sum_{(x_{K'},y_{K'})\in
    \mc{I}_{K_0}}\abs{\gv{g}(x_{K'},y_{K'})-\sum_{j=1}^{14}s_j\gv{\xi}_j}^2.
\end{equation*}
The problem is solved directly by calculating the generalized inverse
of a matrix,
\begin{displaymath}
  [s_1,\cdots,s_j,\cdots,s_{14}]^T = (A^T A)^{-1} A^T b,\quad j=2,3,\cdots,13,
\end{displaymath}
where $A$ and $b$ are defined as follows
\[
A = \begin{bmatrix} 1 & 0 & 0 & x_{K_0} & y_{K_0}& \cdots\\ \vdots &
  \vdots & \vdots & \vdots & \vdots & \cdots\\ 1 & 0 & 0 & x_{K_i} &
  y_{K_i}& \cdots\\ \vdots & \vdots & \vdots & \vdots & \vdots &
  \cdots\\ 1 & 0 & 0 & x_{K_9} & y_{K_9}& \cdots\\ 0 & 1 & x_{K_0}
  &-y_{K_0} & 0 & \cdots\\ \vdots & \vdots & \vdots & \vdots & \vdots
  & \cdots\\ 0 & 1 & x_{K_i} &-y_{K_i} & 0 & \cdots\\ \vdots & \vdots
  & \vdots & \vdots & \vdots & \cdots\\ 0 & 1 & x_{K_9} &-y_{K_9} & 0
  & \cdots
  \end{bmatrix} ,\quad
  b = \begin{bmatrix} g_1(x_{K_0},y_{K_0}) \\ \vdots
    \\ g_1(x_{K_i},y_{K_i}) \\ \vdots \\ g_1(x_{K_9},y_{K_9})
    \\ g_2(x_{K_0},y_{K_0}) \\ \vdots \\ g_2(x_{K_i},y_{K_i})
    \\ \vdots \\ g_2(x_{K_9},y_{K_9})
  \end{bmatrix}
  \quad i=1,2,\cdots,8.
\]
$A$ is a $20\times 14$ matrix, limited by the page space, we only list
the first order part, the rest is easy to be complemented.  $b$ is a
$20\times 1$ vector.

Matrix $(A^T A)^{-1} A^T$ contains all the information of the basis
functions that are defined on element $K_0$, the matrix is relevant
with $\gv{\lambda}_{K_i},i=0,\cdots,9$. The basis function
$\gv{\lambda}_{K_0}$ is determined after the reconstruction are
conducted on each element. Figure \ref{tri_mesh_basis} shows the basis
function $\gv{\lambda}_{K_0}$. We note that the support of
$\gv{\lambda}_{K_0}$ is not equal to the element patch
$S(K_0)$. Insteadly, for any element $K\in\MTh$, the support of the
basis function $\gv{\lambda}_K$ is related with all the element
patches which includes $K$:
\begin{equation}\label{supp}
  \supp(\gv{\lambda}_K)= \bigcup_{K' \in \MTh, K \in S(K')} K'.
\end{equation}
\begin{figure}
  \begin{center}
    \includegraphics[width=0.4\textwidth]{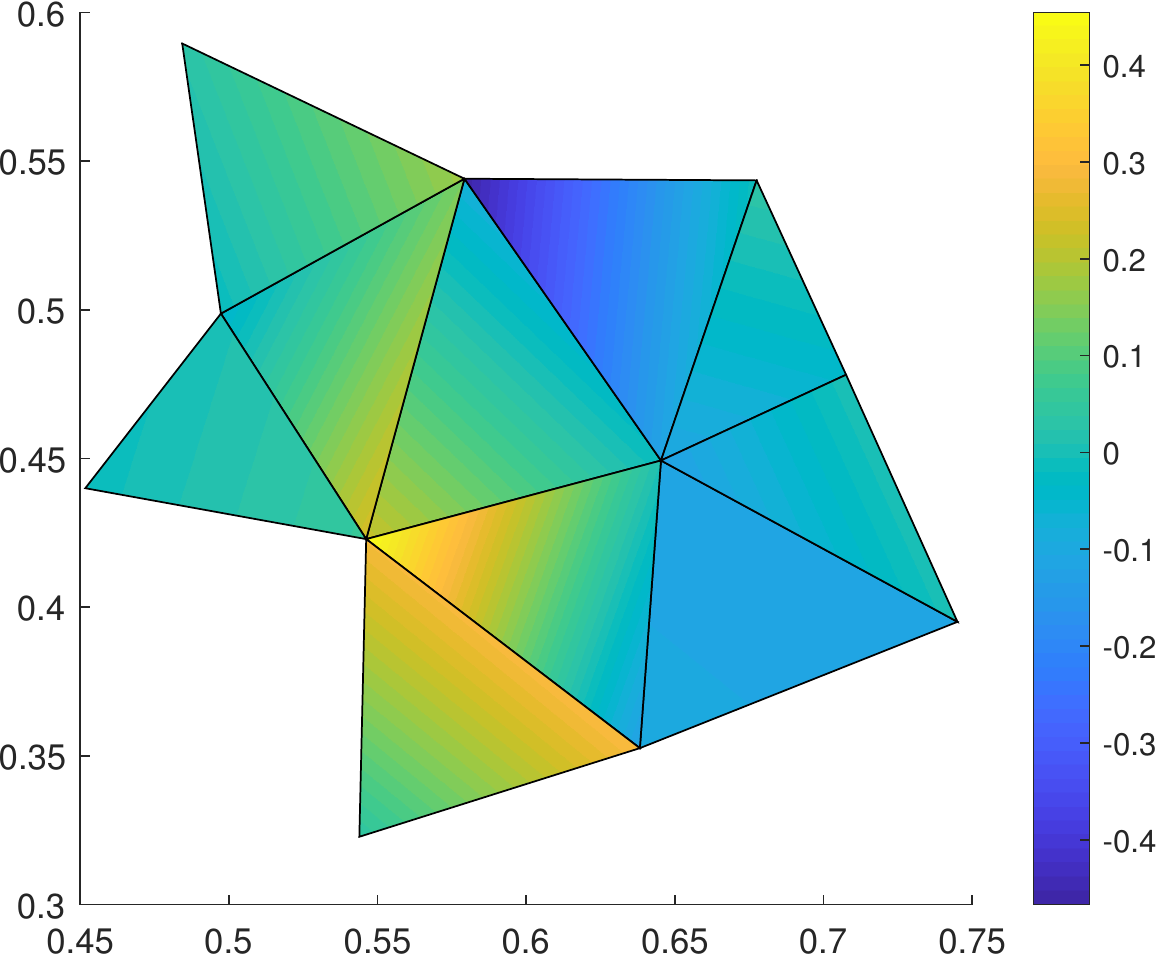}
    \includegraphics[width=0.4\textwidth]{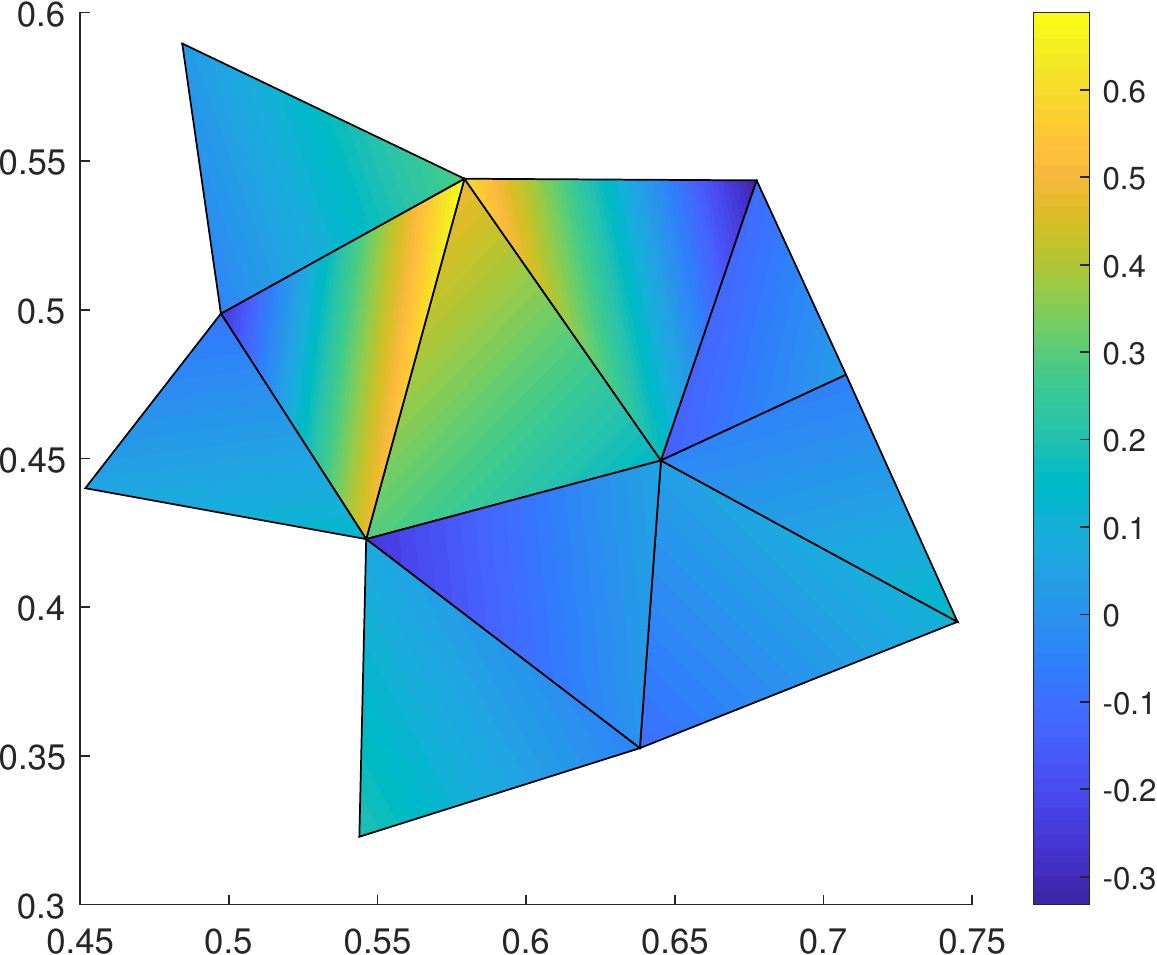}
    \caption{The x component basis function (left)/ y component basis
      function(right).}
  \label{tri_mesh_basis}
  \end{center}
\end{figure}

Due to the discontinuity of the reconstructed finite element space,
the DG method can be directly employed. We end this section by
introducing the average and jump operators which are common used in DG
method. Let $e$ be an interior edge shared by two adjacent elements
$e=\partial K^{+} \cap \partial K^{-}$. The corresponding unit outward
normal vector are denoted by $\gv{\mathrm n}^{+}$,
$\gv{\mathrm n}^{-}$.  Let $v$ and $\gv{v}$ be the scalar valued and
vector valued functions on $\mathcal T_h$, respectively. Following the
traditional DG notations, we define the $average$ operator
$\{ \cdot \}$ on element edges as follows:
\begin{displaymath}
\{v\}=\frac{1}{2}(v^{+} + v^{-}), \quad \{ \gv{v} \} =
\frac{1}{2}(\gv{v}^{+} + \gv{v}^{-}) , \quad\text{on }\ e\in\Gamma^0.
\end{displaymath}
with
$v^+=v|_{K^+},\ v^-=v|_{K^-},\ \gv v^+=\gv v|_{K^+},\ \gv v^-=\gv
v|_{K^-}$. The $jump$ operator $\lj \cdot \rj$ on element edges is as
\begin{displaymath}
  \begin{aligned}
    \lj v \rj =v^{+} \gv{\mathrm n}^{+} + v^{-} \gv{\mathrm n}^{-},
    \quad \lj \gv{v} \rj =\gv{v}^{+}\cdot \gv{\mathrm
      n}^{+}+\gv{v}^{-}\cdot \gv{\mathrm n}^{-}, \\ \lj \gv{v}
    \otimes\gv{\mathrm n} \rj =\gv{v}^{+}\otimes \gv{\mathrm
      n}^{+}+\gv{v}^{-}\otimes \gv{\mathrm n}^{-},\quad
    \text{on}\ e\in\Gamma^0.\\
\end{aligned}
\end{displaymath}
For $e \in \partial \Omega$, we set
\begin{displaymath}
  \begin{aligned}
    \{v\}=v&,\quad \{\gv v\}=\gv v,\quad \lj v \rj=v\gv{\mathrm n},
    \\ \lj \gv v \rj=\gv v\cdot\gv{\mathrm n}&, \quad \lj
    \gv{v}\otimes \gv{\mathrm n} \rj = \gv{v}\otimes \gv{\mathrm n}
    ,\quad\text{on}\ e\in\partial\Omega.\\
  \end{aligned}
\end{displaymath}

\section{Approximation of Stokes Equation}\label{sec:weakform}
We consider the Stokes equation with Dirichlet boundary condition:
\begin{equation}\label{eq:gov}
  \left\{\begin{array}{ll}- \Delta \gv{u} + \nabla p=\gv{f} &
  \mathrm{in}\,\, \Omega, \\[1.5ex] \div \gv{u} = 0 & \mathrm{in }\,\,
  \Omega, \\[1.5ex] \gv{u} = \gv{g} & \mathrm{on }\,\, \partial\Omega,
  \\[1.5ex]
\end{array}\right.
\end{equation}
where $\gv{u}$ is the velocity field, $p$ is the pressure, and
$\gv{g}$ is the given boundary value satisfying the compatibility
condition $(\gv{g}\cdot \gv{n},1)_{\partial \Omega} = 0$. We take the
local divergence-free reconstruction space $\gv{V}_h$ as the trial and
test function spaces. The approximation problem of \eqref{eq:gov} is:
{\it Seek $ \gv{u}_h \in \gv{V}_h$, such that
\begin{equation}\label{eq:weakform}
B_h(\gv{u}_h,\gv{v}_h) = F_h(\gv{v}_h), \quad \forall \gv{v}_h \in
\gv{V}_h,
\end{equation}
where the bilinear form is defined by
\begin{equation}\label{eq:bilinear_operator}
\begin{split}
  B_h(\gv{u}_h,\gv{v}_h) = & \sum_{K\in \mathcal{T}_h} ( \grad
  \gv{u}_h,\grad \gv{v}_h) - (\lj \gv{u}_h \otimes \gv{n} \rj ,\{\grad
  \gv{v}_h\})_{\Gamma_h} - (\{\grad \gv{u}_h\}, \lj \gv{v}_h \otimes
  \gv{n} \rj )_{\Gamma_h} \\ &+\eta (\lj \gv{u}_h \otimes \gv{n}\rj,
  \lj \gv{v}_h \otimes \gv{n}\rj)_{\Gamma_h} +\epsilon (\lj \gv{u}_h
  \cdot \gv{n} \rj , \lj \gv{v}_h \cdot \gv{n}\rj)_{\Gamma_h},
\end{split}
\end{equation}
and
\begin{equation}\label{eq:linear_operator}
F_h(\gv{v}_h) = \sum_{K\in \mathcal{T}_h} ( \gv{f} , \gv{v}_h ) -
(\gv{g}, \grad \gv{v}_h \cdot \gv{n})_{\partial \Omega} + \eta
(\gv{g},\gv{v}_h)_{\partial \Omega} + \epsilon ( \gv{g} \cdot \gv{n} ,
\gv{v}_h \cdot \gv{n})_{\partial \Omega},
\end{equation}
where $\eta$ and $\epsilon$ are penalty parameters, corresponding to
different penalty term, respectively.  }

We note that the weak form \eqref{eq:weakform} is inconsistent
\cite{hansbo2008piecewise}. Precisely, the solution $\gv{u}$ of the
equation \eqref{eq:gov} does not satisfy the weak form
\eqref{eq:weakform}, saying
\[
B_h(\gv{u},\gv{v}_h)\neq F_h(\gv{v}_h) , \quad \forall \gv{v}_h \in
\gv{V}_h.
\]
The inconsistency is caused by the discontinuous of normal component
$\gv{v}_h \cdot \gv{n}$ on interior element edges. Instead of
satisfying the weak form \eqref{eq:weakform}, $\gv{u}$ satisfies
\begin{equation}\label{eq:inconsistent}
B_h(\gv{u},\gv{v}_h)=F_h( \gv{f} , \gv{v}_h) -(p , \lj \gv{v}_h \cdot
\gv{n} \rj )_{\Gamma_h}, \quad \forall \gv{v}_h \in \gv{V}_h
\end{equation}
Therefore, for the interior penalty scheme, the non-consistent penalty
parameter $\epsilon$ need to be taken great enough to control the
consistency error which is similar to the interior penalty method for
elliptic equation \cite{arnold2002unified}. Meanwhile, the penalty
parameter $\eta$ is chosen to guarantee the stability of the operator
$B_h(\cdot,\cdot)$. The magnitude of penalty parameters are taken as
follows,
\begin{equation}\label{penalty}
  \epsilon=\mathrm{O}(h^{-(m+1)}), \quad \eta=\mathrm{O}(h^{-1}),
\end{equation}
where $m$ is the degree of polynomials in $\gv{V}_h$.

Let us define some mesh depended semi-norms, for $\forall \gv{v}_h \in
\gv{V}_h$,
\begin{equation*}
\begin{split}
|\gv{v}_h|_{1,h}^{2} &:= \sum_{K\in
  \mathcal{T}_h}|\gv{v}_h|_{1,K}^{2}, \qquad |\gv{v}_h|_{*}^{2} :=
\sum_{e\in \Gamma_h} h_e^{-1} \|\lj \gv{v}_h \otimes \gv{n}\rj
\|_{L^2(e)}^2 ,\\ \ |\gv{v}_h|_{\diamond}^{2} &:= \sum_{e\in \Gamma_h}
h_e^{-(m+1)}\|\lj \gv{v}_h\cdot \gv{n} \rj \|_{L^2(e)}^2,
\end{split}
\end{equation*}
and introduce the energy norm,
\begin{equation}\label{eq:energy_norm}
  \|\gv{v_h}\|_{h}^2 := |\gv{v}_h|^2_{1,h}+ |\gv{v}_h|^2_{*} +
  |\gv{v}_h|^2_{\diamond}.
\end{equation}
By the results in the previous section, we instantly have following
approximation estimate about energy norm,
\begin{lemma}\label{lemma:approximate_energy_error}
  For $\gv{u}\in [H^{m+1}(\Omega)]^2$ to be the solution of equation
  \eqref{eq:gov}, and $\gv{u}_I\in \gv{V}_h$ to be the interpolation
  of $\gv{u}$, there exists a constant $C$ that depends on $N$,
  $\sigma$, $\gamma$ and $m$, such that
\begin{equation}\label{eq:approx_energy_error}
\|\gv{u}-\gv{u}_I\|_h \leq C
(h^{\frac{m}{2}}+\Lambda_{m}d^{\frac{m}{2}}) | \gv{u}
|_{H^{m+1}(\Omega)} .
\end{equation}
\end{lemma}
Here we need to clarify that the approximation error is not optimal,
which is due to the fact that the BDM space is not a subspace of the
approximation space $\gv{V}_h$. The third term of energy norm is
dominant in the approximation error. This fact makes the numerical
solution not able to attain the optimal accuracy order.

For the consistency error, we have following estimate,
\begin{lemma}\label{lemma:inconsistent}
  For $\forall \gv{v}_{h} \in \gv{V}_h$ and $p\in H^1(\Omega)$, we
  have
\begin{equation}\label{eq:inconsistent_error}
\sum_{e\in \Gamma_h} \int_{e} p \lj \gv{v}_{h} \cdot \gv{n} \rj\,ds
\leq C h^\frac{m}{2} \|p\|_{H^{1}(\Omega)} \|\gv{v}_{h}\|_h.
\end{equation}
\end{lemma}
\begin{proof}
By {\em Agmon inequality}~\eqref{eq:agmon}, we have
\begin{equation*}
\begin{split}
\sum_{e\in \Gamma_h} \int_{e}p \lj \gv{v}_{h} \cdot \gv{n} \rj \,ds &
\leq \sum_{e\in \Gamma_h} \int_{e} h^{\frac{m+1}{2}} p
h^{-\frac{m+1}{2}} \lj \gv{v}_{h} \cdot \gv{n} \rj \,ds \\ & \leq C
h^{\frac{m}{2}} \sum_{K \in \mathcal{T}_h} \Lr{ \|p\|_{L^{2}(K)}^2+h^2
  \|\nabla p\|_{K}^2}^{\frac{1}{2}} \sum_{e\in \Gamma_h}
\Lr{h^{-(m+1)} \|\lj \gv{v}_{h} \cdot \gv{n} \rj\|_{L^2(e)}^2
}^{\frac{1}{2}} \\ & \leq C h^{\frac{m}{2}} \|p\|_{H^{1}(\Omega)}
\|\gv{v}_{h}\|_h
\end{split}
\end{equation*}

\end{proof}
Next, the boundedness and the stability of the operator
$B_h(\cdot,\cdot)$ can be claimed as
\begin{lemma}\label{bounded}
For $\forall \gv{u}_{h} \in \gv{V}_h$ , $\forall \gv{v}_{h} \in
\gv{V}_h$, and sufficiently large $\eta$ and $\epsilon$ we have
\begin{equation}
\begin{split}
  B_h(\gv{u}_{h},\gv{v}_{h})& \leq C \|\gv{u}_{h}\|_h\|\gv{v}_{h}\|_h
  ,\\ B(\gv{v}_{h},\gv{v}_{h}) & \geq C \|\gv{v}_{h}\|_h^2.
\end{split}
\end{equation}
\end{lemma}

Now, with the above lemmas, we are ready to give the error estimate of
the numerical solution \eqref{eq:weakform},
\begin{theorem}\label{thm:energy_error}
  For $\gv{u}\in [H^{m+1}(\Omega)]^2$, $p\in H^1(\Omega)$ to be the
  solution of equation \eqref{eq:gov}, and $\gv{u}_h\in \gv{V}_h$ to
  be the solution of equation \eqref{eq:weakform}, we then have
\begin{equation}\label{energy_error_estimate}
\|\gv{u}-\gv{u}_h\|_h \leq C h^{\frac{m}{2}} \Lr{| \gv{u}
  |_{H^{m+1}(\Omega)}+ \|p\|_{H^{1}(\Omega)}}.
\end{equation}
Furthermore, the $L^2$ error estimate is as
\begin{equation}\label{L2_error_estimate}
\begin{split}
\|\gv{u} - \gv{u}_h \|_{L^2(\Omega)} &\leq C h \Lr{| \gv{u}
  |_{H^{m+1}(\Omega)}+ \|p\|_{H^{1}(\Omega)}},\ m=1, \\ \|\gv{u} -
\gv{u}_h \|_{L^2(\Omega)} &\leq C h^{\frac{m}{2}+1}\Lr{ | \gv{u}
  |_{H^{m+1}(\Omega)}+ \|p\|_{H^{1}(\Omega)}}, \ m\geq 2 .
\end{split}
\end{equation}
\end{theorem}
\begin{proof}
We split the error with an interpolation function $\gv{u}_I$ in
$\gv{V}_h$,
\begin{equation}\label{eq:tri}
\|\gv{u}-\gv{u}_h\|_h \leq \|\gv{u}-\gv{u}_I\|_h +
\|\gv{u}_I-\gv{u}_h\|_h.
\end{equation}
Then together with the consistency error, we have
\begin{equation}
B_h(\gv{u}-\gv{u}_h,\gv{v}_h) = (p , \lj \gv{v}_h \cdot \gv{n} \rj
)_{\Gamma_h}, \quad \forall \gv{v}_h \in \gv{V}_h
\end{equation}
From the stability of bilinear operator $B_h(\cdot,\cdot)$, the second
term in \eqref{eq:tri} has
\begin{equation}\label{error}
\begin{split}
C \|\gv{u}_I-\gv{u}_h\|^2_h\leq &
B_h(\gv{u}_I-\gv{u}_h,\gv{u}_I-\gv{u}_h) \\ = &
B_h(\gv{u}_I-\gv{u},\gv{u}_I-\gv{u}_h) +
B_h(\gv{u}-\gv{u}_h,\gv{u}_I-\gv{u}_h) \\ = & B_h
(\gv{u}_I-\gv{u},\gv{u}_I-\gv{u}_h) + \left(p, \lj \gv{n} \cdot
(\gv{u}_I-\gv{u}_h) \rj \right )_{\Gamma_h^0} \\ \leq & C
\|\gv{u}-\gv{u}_I\|_h \|\gv{u}_I-\gv{u}_h\|_h + C h^{\frac{m}{2}}
\|p\|_{H^{1}(\Omega)} \|\gv{u}_{I}-\gv{u}_h\|_h .
\end{split}
\end{equation}
Collecting the above estimates and Lemma \ref{lemma:inconsistent}, it
gives
\begin{equation}\label{eq:err_interp_numerial}
\|\gv{u}_I-\gv{u}_h\|_h \leq C h^{\frac{m}{2}}\Lr{ | \gv{u}
  |_{H^{m+1}(\Omega)}+ \|p\|_{H^{1}(\Omega)}}.
\end{equation}
The first term in \eqref{eq:tri} is the interpolation error,
\begin{equation}\label{eq:interp}
\|\gv{u} - \gv{u}_I\|_{h} \leq C h^{\frac{m}{2}}
|\gv{u}|_{H^{m+1}(\Omega)}.
\end{equation}
Collecting estimates \eqref{eq:tri} ,\eqref{eq:err_interp_numerial}
and \eqref{eq:interp} together, we can have
\eqref{energy_error_estimate}.

For the $L^2$ error estimate, we define the auxiliary functions
$(\gv{\varphi} , q)$ and the adjoint problem
\begin{equation}\label{eq:adjoint}
\left\{\begin{array}{ll}- \Delta \gv{\varphi} + \nabla
q=\gv{u}-\gv{u}_h & \mathrm{in }\,\, \Omega, \\[1.5ex] \div
\gv{\varphi} = 0 & \mathrm{in }\,\, \Omega, \\[1.5ex] \gv{\varphi} = 0
& \mathrm{on }\,\, \partial\Omega, \\[1.5ex]
\end{array}\right.
\end{equation}
then apply the test function $\gv{u}-\gv{u}_h$, we have
\begin{equation}\label{eq:l2_split}
\begin{split}
(\gv{u}-\gv{u}_h,\gv{u}-\gv{u}_h)_{\Omega}=&
  B_h(\gv{\varphi},\gv{u}-\gv{u}_h) + (q , \lj (\gv{u}-\gv{u}_h) \cdot
  \gv{n} \rj )_{\Gamma_h}.
\end{split}
\end{equation}
The regularity of Stokes equations implies
\begin{equation}
\begin{split}
\|q\|_{H^1(\Omega)} &\leq \|\gv{u} - \gv{u}_h \|_{L^2(\Omega)},
\\ \|\gv{\varphi}\|_{H^2(\Omega)} &\leq \|\gv{u} - \gv{u}_h
\|_{L^2(\Omega)}, \\ \|\gv{\varphi} - \gv{\varphi}_{BDM} \|_{h} &\leq
h\|\gv{u} - \gv{u}_h \|_{L^2(\Omega)}.
\end{split} 
\end{equation}
where $\gv{\varphi}_{BDM}$ is the BDM interpolation function. This
gives
\begin{equation}\label{eq:l2_split_t1}
\begin{split}
  B_h(\gv{\varphi},\gv{u}-\gv{u}_h) &=
  B_h(\gv{\varphi}-\gv{\varphi}_{BDM},\gv{u}-\gv{u}_h) \\ &\leq C
  \|\gv{u}-\gv{u}_h\|_{h} \|\gv{\varphi}-\gv{\varphi}_{BDM}\|_{h}
  \\ &\leq C h \|\gv{u}-\gv{u}_h\|_{h} \|\gv{u} - \gv{u}_h
  \|_{L^2(\Omega)}
\end{split} 
\end{equation}
and
\begin{equation}\label{eq:l2_split_t2}
\begin{split}
(q , \lj (\gv{u}-\gv{u}_h) \cdot \gv{n} \rj )_{\Gamma_h} &\leq C
  h^{\frac{m}{2}} \|q\|_{H^1(\Omega)}\|\gv{u}-\gv{u}_h\|_{h} \\ &\leq
  C h^{\frac{m}{2}} \|\gv{u}-\gv{u}_h\|_{h} \|\gv{u} - \gv{u}_h
  \|_{L^2(\Omega)}.
\end{split} 
\end{equation}
Inserting \eqref{eq:l2_split_t1} and \eqref{eq:l2_split_t2} to the
equation \eqref{eq:l2_split},
\begin{equation}\label{eq:l2_split_t3}
\|\gv{u}-\gv{u}_h\|_{L^2(\Omega)} \leq C h \|\gv{u}-\gv{u}_h\|_{h} + C
h^{\frac{m}{2}} \|\gv{u}-\gv{u}_h\|_{h}.
\end{equation}
Then substituting the energy norm estimate
\eqref{energy_error_estimate} into \eqref{eq:l2_split_t3}, the $L^2$
error estimate \eqref{L2_error_estimate} is obtained.
\end{proof}


\section{Numerical Examples}\label{sec:examples}
We present some numerical examples to illustrate the
effectiveness of the proposed method. One merit of the method is the
linear system is maintained with the given mesh while ignore the
increase of approximation order. The maximum degree of approximation
polynomial is taken as $4$ which is limited by the condition number
of linear system. The numerical examples with various polygonal meshes
are demonstrated. A direct solver is used to solve the
linear algebra systems after discretisation.
\subsection{Analytical example}
This is an example with a smooth solution to verify the
convergence order as the error estimate predicted. The
Stokes problem with Dirichlet boundary condition is solved in two
dimensional square domain $\Omega=[0,1]\times[0,1]$. The exact
velocity field and pressure are given as follows,
\begin{align*}
  \gv{u} =&\left(\begin{array}{ll} \sin(2\pi x)*\cos (2\pi y) \\ -
    \cos(2\pi x)*\sin (2\pi y)
   \end{array}\right),\\
    p= &x^2+y^2.
\end{align*}
The body force $\gv{f} $ and boundary condition $\gv{g}$ are given
correspondingly. The quasi-uniform triangular and quadrilateral meshes
are used, both of them are generated by the software {\it
  gmsh}\cite{geuzaine2009gmsh}.
\begin{table}[]
  \begin{center}
  \scalebox{0.9}{
    \begin{tabular}{|c|c|c|c||c||c||c||c||c||c||c||}
      \hline & & h=1.0E-1 & 5.0E-2 & & 2.5E-2 & & 1.25E-2 & & 6.25E-3
      & \\ \cline{3-4} \cline{6-6} \cline{8-8} \cline{10-10}
      \multirow{-2}{*}{$P^{m}$} & \multirow{-2}{*}{Norm} & error &
      error & \multirow{-2}{*}{order} & error &
      \multirow{-2}{*}{order} & error & \multirow{-2}{*}{order} &
      error & \multirow{-2}{*}{order} \\ \hline \hline

      & $\|\cdot\|_{L^{2}}$ & 5.41E-02 & 2.41E-02 & 1.16 & 9.97E-03 &
      1.27 & 4.20E-03 & 1.24 & 1.92E-03 & 1.13 \\ \cline{2-11} &
      \multicolumn{1}{c|}{$\|\cdot\|_{1,h}$} &
      \multicolumn{1}{c|}{1.50E+00} & 8.76E-01 & 0.77 & 4.61E-01 &
      0.92 & 2.55E-01 & 0.85 & 1.48E-01 & 0.77 \\ \cline{2-11}
      \multirow{-3}{*}{1} & \multicolumn{1}{c|}{$\|\cdot\|_{h}$} &
      \multicolumn{1}{c|}{2.10E+00} & 1.39E+00 & 0.59 & 8.16E-01 &
      0.77 & 4.97E-01 & 0.71 & 3.15E-01 & 0.65 \\ \hline \hline

      & $\|\cdot\|_{L^{2}}$ & 1.47E-02 & 2.88E-03 & 2.34 & 6.79E-04 &
      2.08 & 1.72E-04 & 1.97 & 4.47E-05 & 1.94 \\ \cline{2-11} &
      \multicolumn{1}{c|}{$\|\cdot\|_{1,h}$} &
      \multicolumn{1}{c|}{3.97E-01} & 1.19E-01 & 1.73 & 3.78E-02 &
      1.66 & 1.28E-02 & 1.55 & 4.84E-03 & 1.40 \\ \cline{2-11}
      \multirow{-3}{*}{2} & \multicolumn{1}{c|}{$\|\cdot\|_{h}$} &
      \multicolumn{1}{c|}{7.72E-01} & 3.52E-01 & 1.13 & 1.55E-01 &
      1.17 & 7.53E-01 & 1.04 & 3.71E-03 & 1.01 \\ \hline \hline

      & $\|\cdot\|_{L^{2}}$ & 5.84E-03 & 8.18E-04 & 2.83 & 1.06E-04 &
      2.94 & 1.88E-05 & 2.49 & 2.97E-06 & 2.66 \\ \cline{2-11} &
      \multicolumn{1}{c|}{$\|\cdot\|_{1,h}$} &
      \multicolumn{1}{c|}{1.41E-01} & 2.67E-02 & 2.40 & 5.30E-03 &
      2.33 & 1.48E-03 & 1.83 & 3.54E-04 & 2.06 \\ \cline{2-11}
      \multirow{-3}{*}{3} & \multicolumn{1}{c|}{$\|\cdot\|_{h}$} &
      \multicolumn{1}{c|}{3.33E-01} & 9.47E-01 & 1.81 & 2.66E-02 &
      1.83 & 9.20E-03 & 1.53 & 3.04E-03 & 1.59 \\ \hline \hline

      & $\|\cdot\|_{L^{2}}$ & 1.14E-03 & 1.03E-04 & 3.46 & 8.25E-06 &
      3.64 & 1.08E-06 & 2.92 & 1.36E-07 & 2.99 \\ \cline{2-11} &
      \multicolumn{1}{c|}{$\|\cdot\|_{1,h}$} &
      \multicolumn{1}{c|}{3.24E-02} & 4.06E-03 & 2.99 & 4.99E-04 &
      3.02 & 9.96E-05 & 2.32 & 1.77E-05 & 2.48 \\ \cline{2-11}
      \multirow{-3}{*}{4} & \multicolumn{1}{c|}{$\|\cdot\|_{h}$} &
      \multicolumn{1}{c|}{9.44E-02} & 1.62E-02 & 2.54 & 3.76E-03 &
      2.10 & 9.15E-04 & 2.04 & 2.12E-04 & 2.10 \\ \hline

  \end{tabular}}
  \caption{Numerical errors on the quasi-uniform triangle meshes for
    Example 1.}
  \label{tri_mesh_table}
  \end{center}
\end{table}

\begin{figure}
  \begin{center}
    \includegraphics[width=0.3\textwidth]{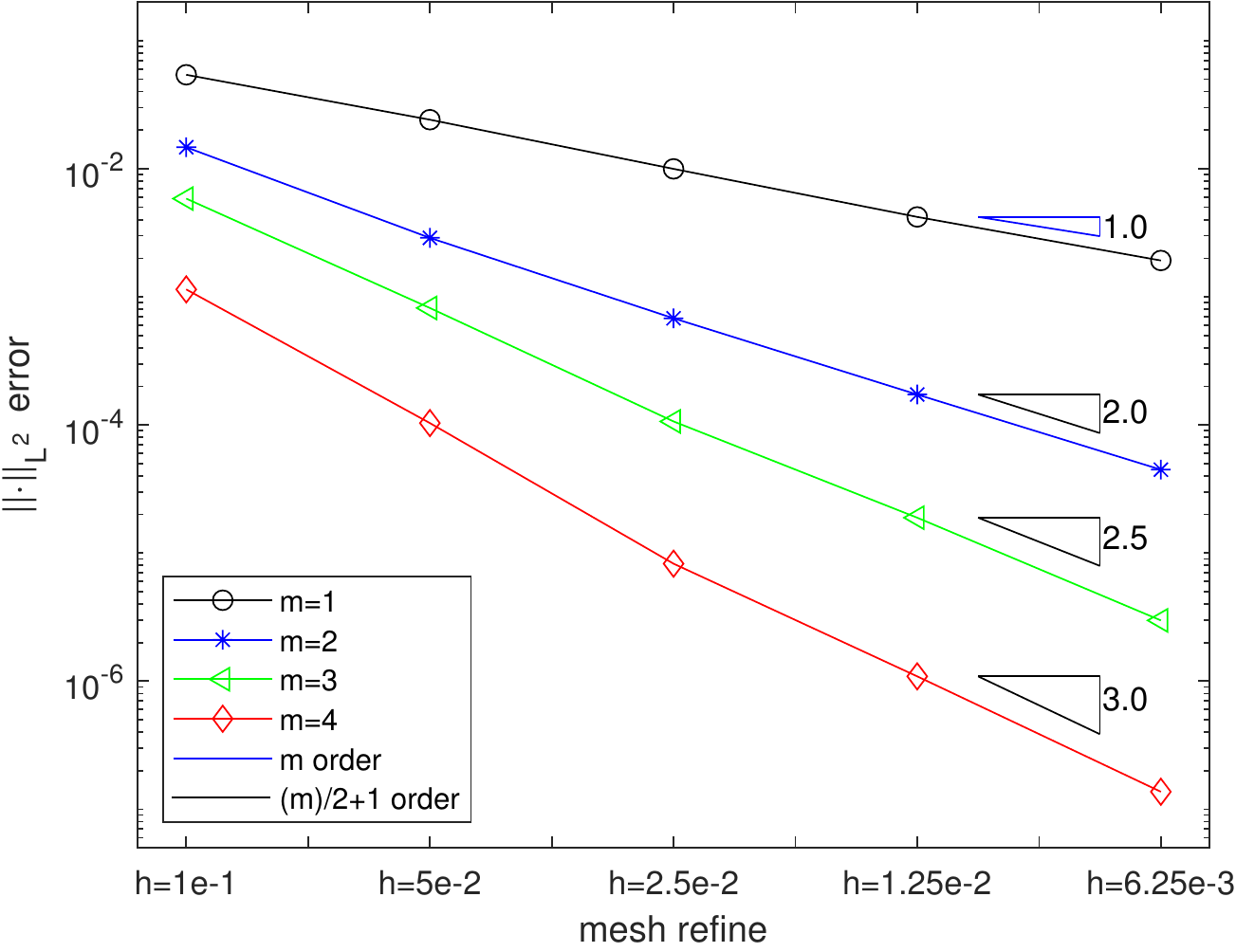}
    \includegraphics[width=0.3\textwidth]{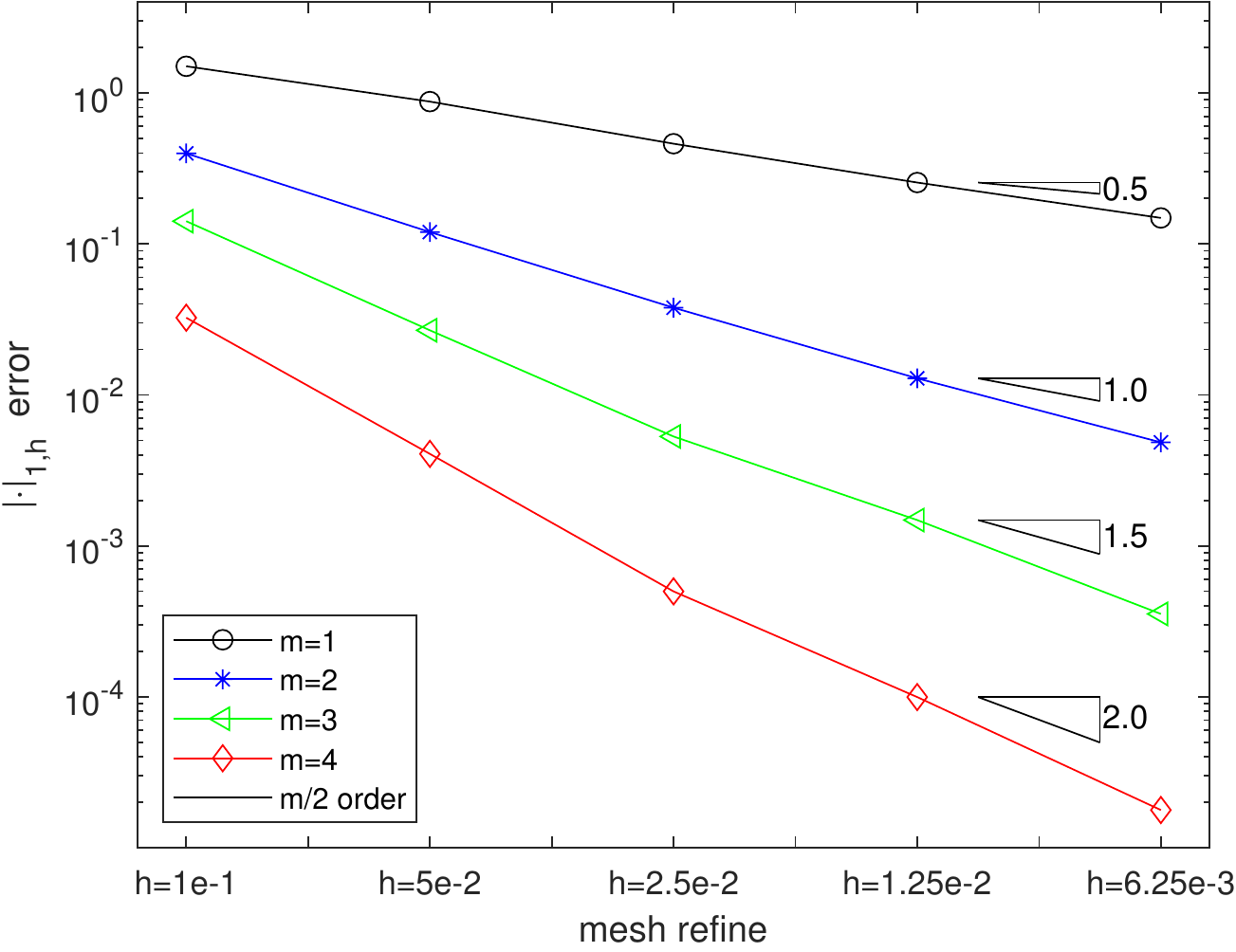}
    \includegraphics[width=0.3\textwidth]{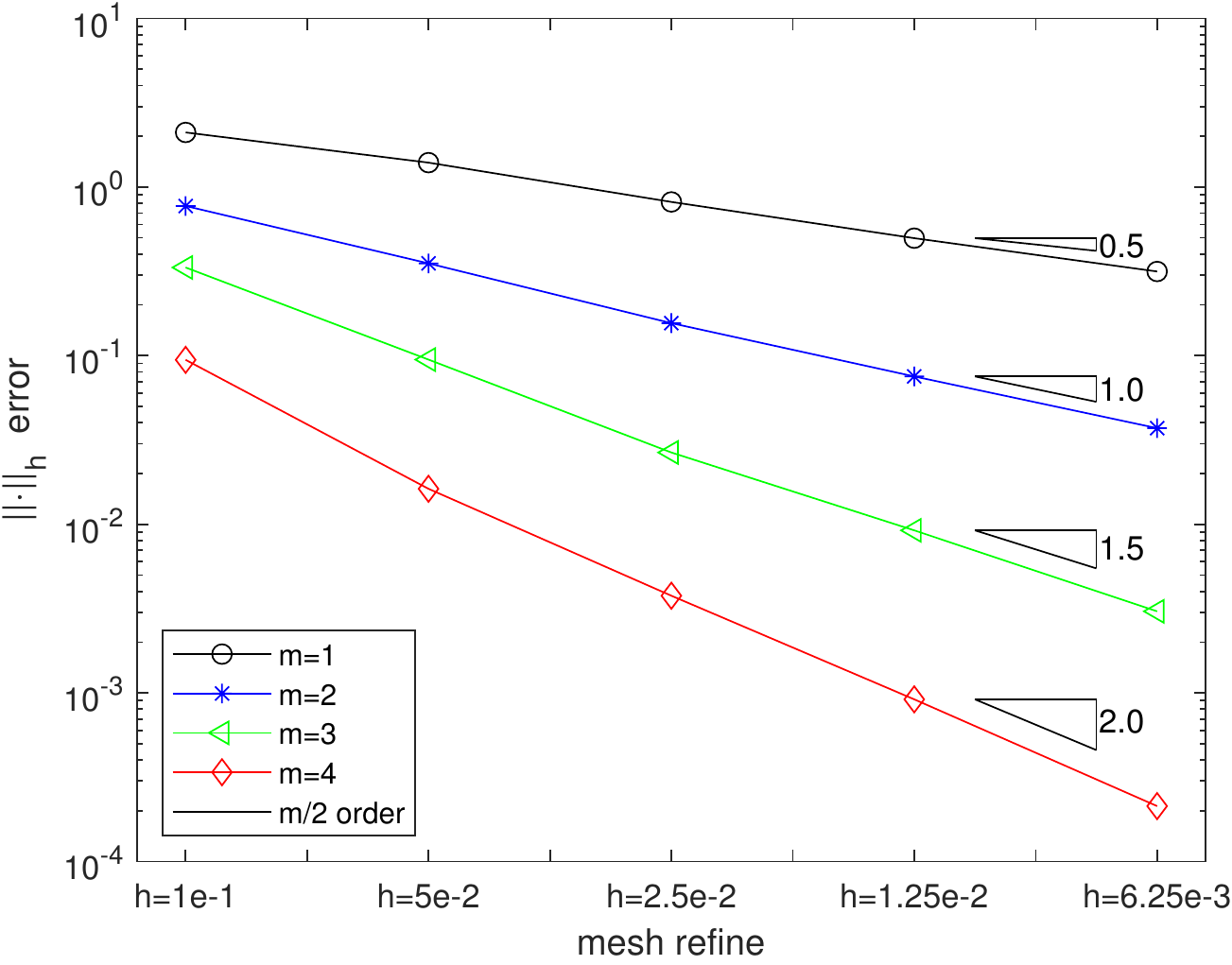}
    \caption{The accuracy order on triangle meshes for Example 1.}
  \label{tri_mesh_fig}
  \end{center}
\end{figure}

\begin{table}[]
  \begin{center}
  \scalebox{0.9}{
    \begin{tabular}{|c|c|c|c||c||c||c||c||c||c||c||}
      \hline & & h=1.0E-1 & 5.0E-2 & & 2.5E-2 & & 1.25E-2 & & 6.25E-3
      & \\ \cline{3-4} \cline{6-6} \cline{8-8} \cline{10-10}
      \multirow{-2}{*}{$P^{m}$} & \multirow{-2}{*}{Norm} & error &
      error & \multirow{-2}{*}{order} & error &
      \multirow{-2}{*}{order} & error & \multirow{-2}{*}{order} &
      error & \multirow{-2}{*}{order} \\ \hline \hline 

      & $\|\cdot\|_{L^{2}}$ & 1.03E-01 & 4.41E-02 & 1.22 & 1.74E-02 &
      1.33 & 8.14E-03 & 1.10 & 3.67E-03 & 1.14 \\ \cline{2-11} &
      \multicolumn{1}{c|}{$\|\cdot\|_{1,h}$} &
      \multicolumn{1}{c|}{1.97E+00} & 1.08E+00 & 0.86 & 5.54E-01 &
      0.96 & 2.96E-01 & 0.90 & 1.62E-01 & 0.87 \\ \cline{2-11}
      \multirow{-3}{*}{1} & \multicolumn{1}{c|}{$\|\cdot\|_{h}$} &
      \multicolumn{1}{c|}{3.10E+00} & 1.92E+00 & 0.69 & 1.22E+00 &
      0.64 & 8.27E-01 & 0.56 & 5.50E-01 & 0.58 \\ \hline \hline 

      & $\|\cdot\|_{L^{2}}$ & 3.80E-02 & 7.71E-03 & 2.30 & 1.92E-03 &
      1.99 & 4.97E-04 & 1.95 & 1.21E-04 & 2.03 \\ \cline{2-11} &
      \multicolumn{1}{c|}{$\|\cdot\|_{1,h}$} &
      \multicolumn{1}{c|}{7.95E-01} & 2.06E-01 & 1.94 & 7.11E-02 &
      1.53 & 2.81E-02 & 1.33 & 1.10E-02 & 1.35 \\ \cline{2-11}
      \multirow{-3}{*}{2} & \multicolumn{1}{c|}{$\|\cdot\|_{h}$} &
      \multicolumn{1}{c|}{1.48E+00} & 5.82E-01 & 1.35 & 2.63E-01 &
      1.14 & 1.22E-01 & 1.10 & 5.58E-03 & 1.13 \\ \hline \hline 

      & $\|\cdot\|_{L^{2}}$ & 1.11E-02 & 1.88E-03 & 2.56 & 2.75E-04 &
      2.77 & 3.94E-05 & 2.80 & 5.96E-06 & 2.72 \\ \cline{2-11} &
      \multicolumn{1}{c|}{$\|\cdot\|_{1,h}$} &
      \multicolumn{1}{c|}{2.59E-01} & 3.31E-02 & 2.96 & 6.17E-03 &
      2.42 & 1.74E-03 & 1.82 & 4.92E-04 & 1.82 \\ \cline{2-11}
      \multirow{-3}{*}{3} & \multicolumn{1}{c|}{$\|\cdot\|_{h}$} &
      \multicolumn{1}{c|}{4.11E-01} & 9.71E-02 & 2.08 & 3.06E-02 &
      1.66 & 1.10E-02 & 1.46 & 3.82E-03 & 1.53 \\ \hline \hline 

      & $\|\cdot\|_{L^{2}}$ & 8.18E-03 & 6.85E-04 & 3.57 & 5.32E-05 &
      3.68 & 4.15E-06 & 3.68 & 4.04E-07 & 3.35 \\ \cline{2-11} &
      \multicolumn{1}{c|}{$\|\cdot\|_{1,h}$} &
      \multicolumn{1}{c|}{1.59E-01} & 9.69E-03 & 4.03 & 1.26E-03 &
      2.94 & 2.18E-04 & 2.52 & 4.06E-05 & 2.42 \\ \cline{2-11}
      \multirow{-3}{*}{4} & \multicolumn{1}{c|}{$\|\cdot\|_{h}$} &
      \multicolumn{1}{c|}{2.52E-01} & 3.55E-02 & 2.82 & 8.73E-03 &
      2.02 & 2.25E-03 & 1.95 & 5.62E-04 & 2.00 \\ \hline

  \end{tabular}}
  \caption{Numerical errors on the quasi-uniform quadrilateral meshes
    for Example 1.}
  \label{quad_mesh_table}
  \end{center}
\end{table}

\begin{figure}
  \begin{center}
    \includegraphics[width=0.3\textwidth]{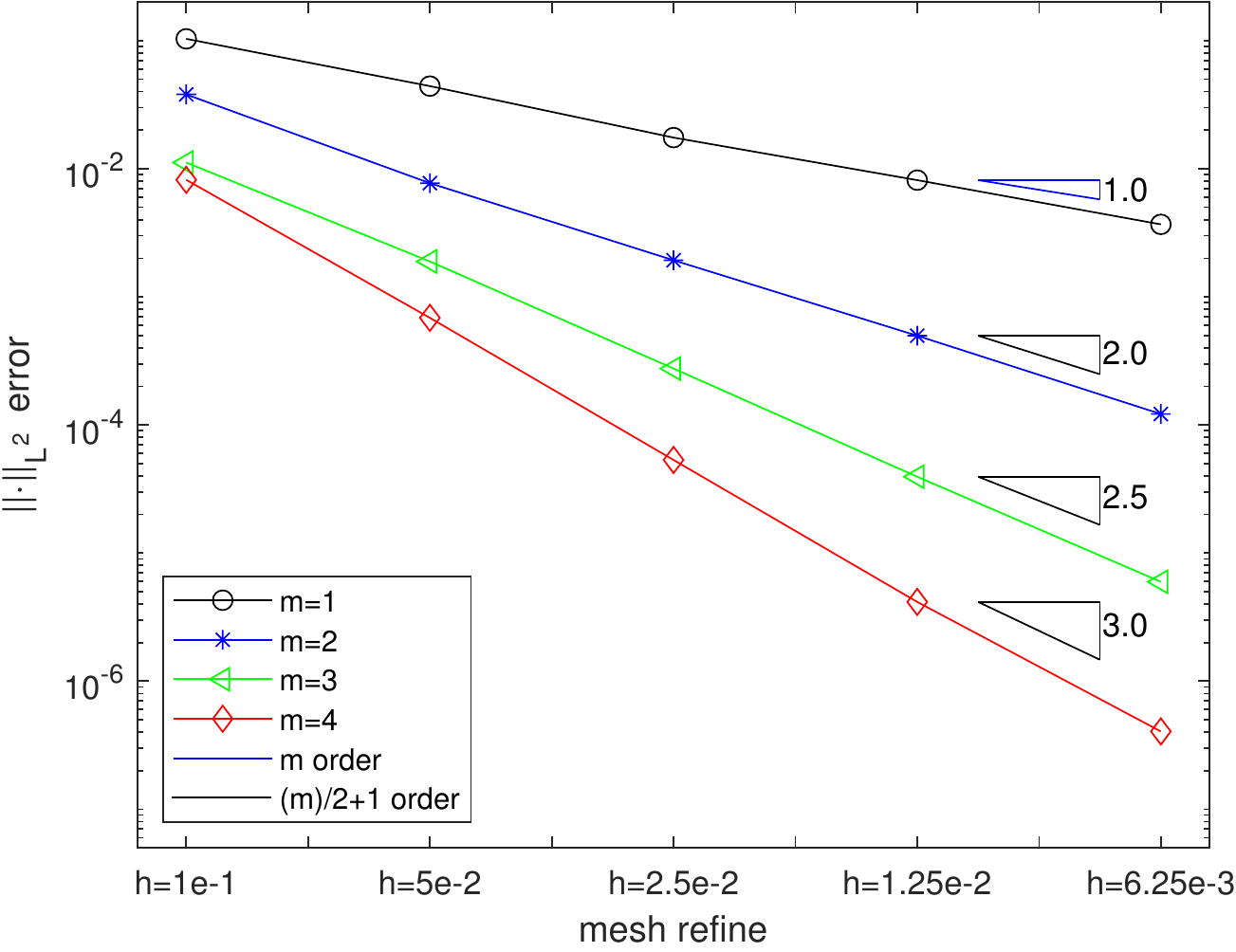}
    \includegraphics[width=0.3\textwidth]{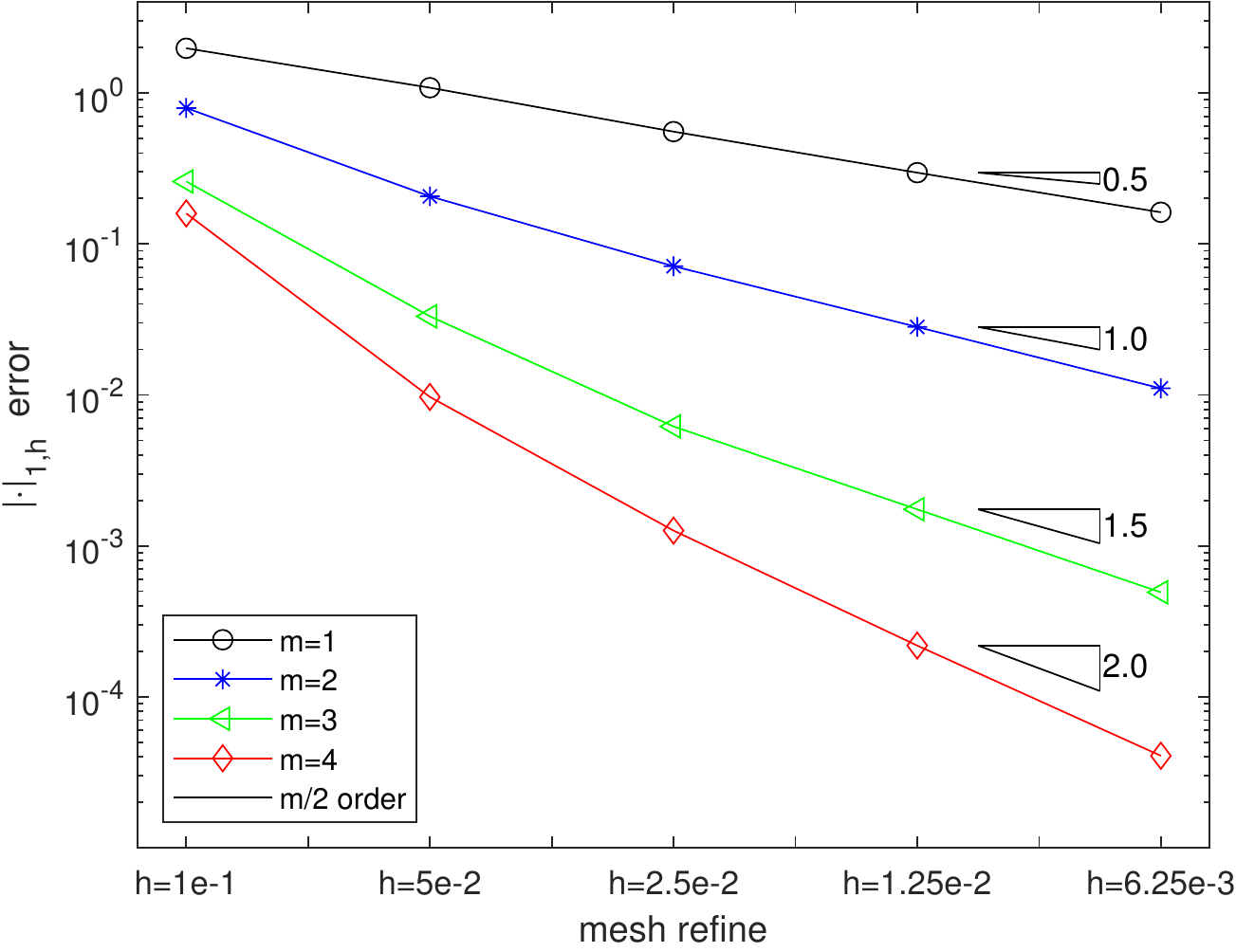}
    \includegraphics[width=0.3\textwidth]{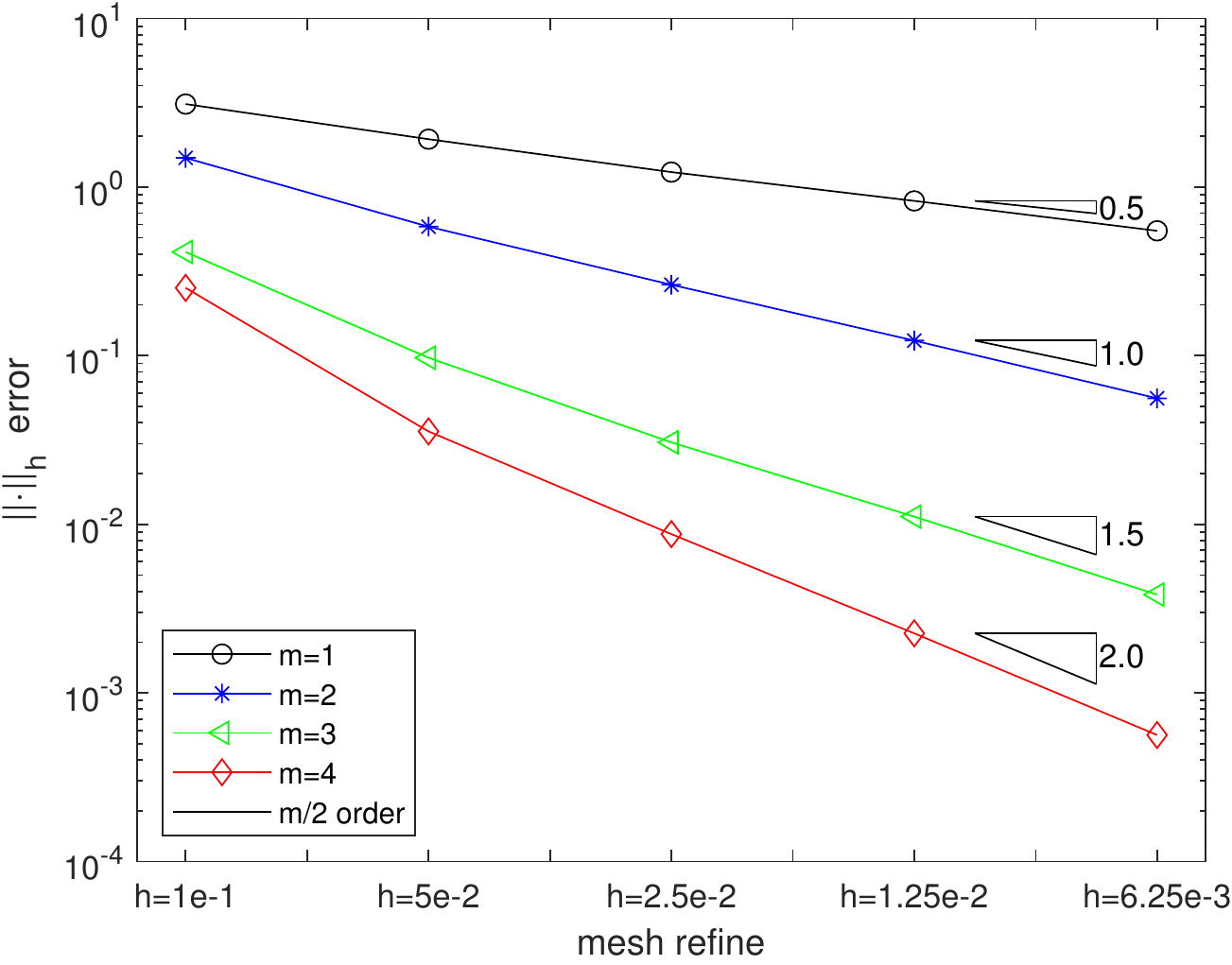}
    \caption{The accuracy order on quadrilateral meshes for Example 1.}
  \label{quad_mesh_fig}
  \end{center}
\end{figure}

Table \ref{tri_mesh_table} and \ref{quad_mesh_table} present the error
of the numerical solution in $L^2$ norm, $|\cdot|_{1,h}$ seminorm and
$\|\cdot\|_h$ norm. Figure \ref{tri_mesh_fig} and \ref{quad_mesh_fig}
plot the accuracy order under different norm, respectively.  The
behavior of the proposed method agrees perfectly with the theoretical
analysis on the convergence order to be $h^{\frac{m}{2}}$ on
$\|\cdot\|_h$ norm. And the convergence order of $L^2$ norm and
$|\cdot|_{1,h}$ seminorm are not optimal as usual situations, which
are $m+1$ and $m$, respectively. This is due to the fact that the
space $\gv{V}_h$ is not large enough that it can not contain the $BDM$
space as a subspace as the traditional discontinued Galerkin
space. The difference between $|\cdot|_{1,h}$ seminorm and
$\|\cdot\|_h$ norm error implies that the seminorm
$|\cdot|_{\diamond}$ is dominate in energy norm. The condition number
of resulting matrix is mainly determined by the penalty parameter
$\epsilon$. And the exponential growth of $\epsilon$ with the
increasing of the degree of the polynomial results in very
ill-posed linear system.

\subsection{Lid-driven cavity example}
The lid-driven cavity is a benchmark test for the incompressible flow
that does not have an exact solution. Here we sets the cavity domain
as a rectangle $\Omega=[0,1]\times[0,1.5]$. The body force is zero,
and a Dirichlet boundary condition is given as
\begin{equation}
\gv{g}=\left\{\begin{array}{ll} (1,0) \quad \text{if} \ x
\in(0,1),y=1.5,\\ (0,0) \quad \text{other boundaries.}
\end{array}\right.
\end{equation} 

\begin{figure}
  \begin{center}
    \includegraphics[width=0.31\textwidth]{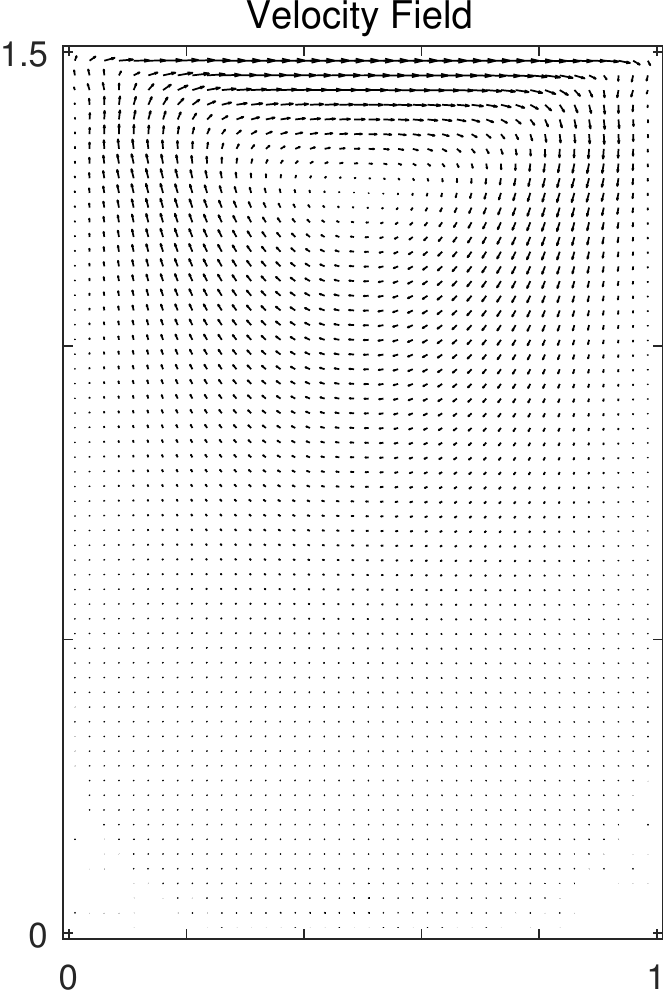}
    \includegraphics[width=0.31\textwidth]{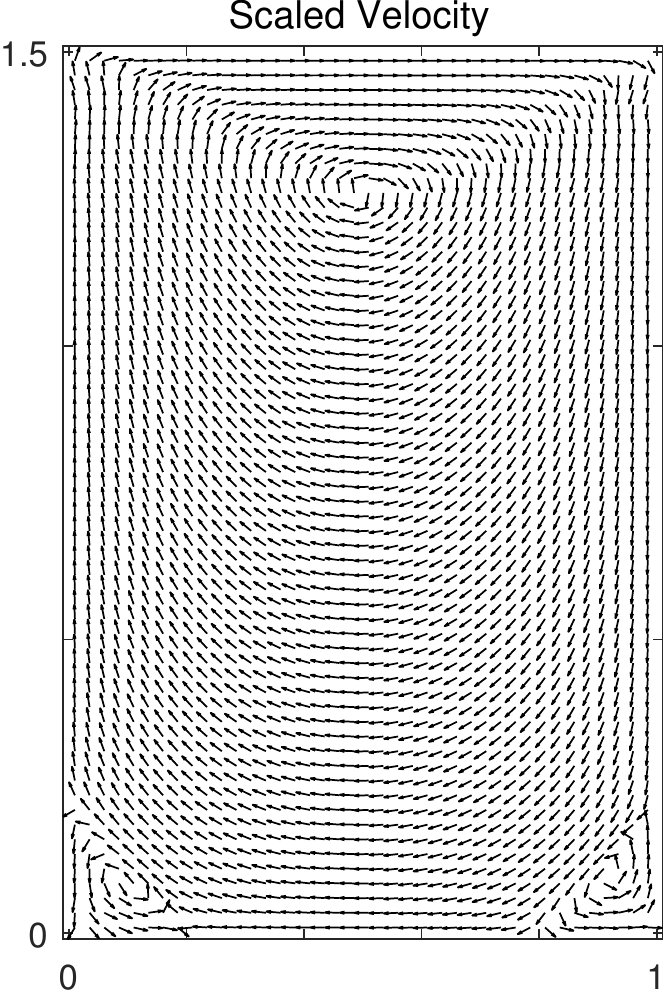}
    \includegraphics[width=0.31\textwidth]{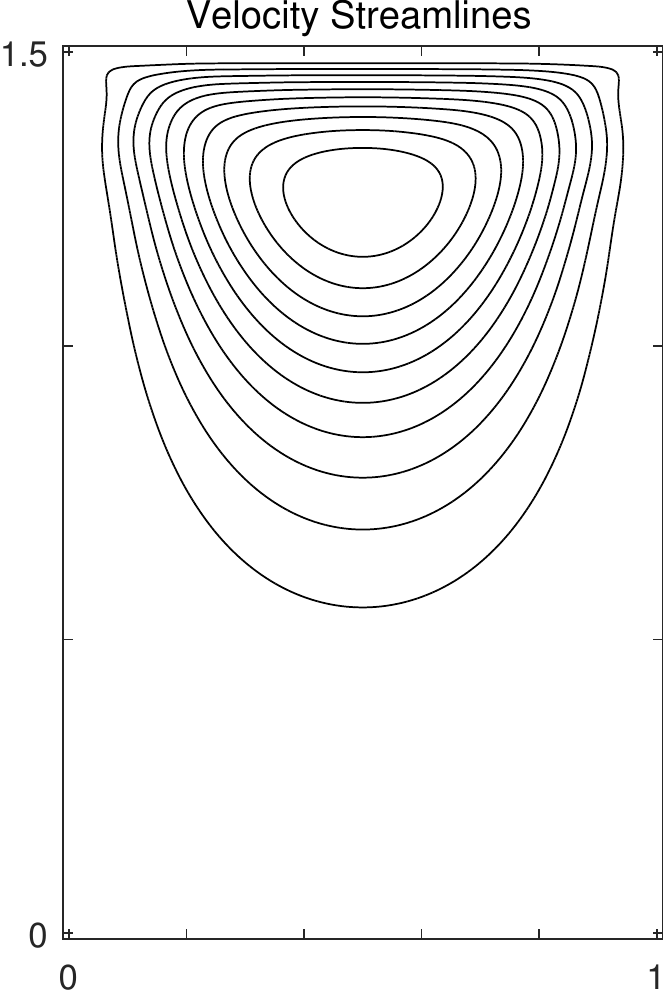}
    \caption{The velocity, scaled velocity field and streamline of
      lid-driven cavity}
  \label{lid_driven_fig}
  \end{center}
\end{figure}
Figure \ref{lid_driven_fig} shows the velocity fields and the scaled
velocity. The domain is segmented by a quasi-uniform triangle mesh
meanwhile the third order polynomial is employed. The numerical
results fit the realities. The scaled velocity present the expected
phenomenon which the Contra vortices appear in the bottom corner. The
subtle vortices will capture with refined mesh.

\subsection{Flow around cylinder}
Flow around cylinder is another representative benchmark test for the
incompressible flow. The numerical setting is slightly different from
the traditional setting. The polygonal elements are considered to
demonstrate the compatibility of the proposed method. While the
computational domain is a rectangle minus a circle
$\Omega=[0,1.5]\times[0,1] / (x-0.5)^2+(y-0.5)^2 \leq 0.2^2$, which is
partitioned into polygonal elements mesh by PolyMesher
\cite{talischi2012polymesher}. The Dirichlet boundary condition is
given while with zero body force.
\begin{equation}
\gv{g}=\left\{\begin{array}{ll} (y(1-y),0) \quad \text{if} \ x=0,1.5,y
\in(0,1),\\ (0,0) \quad \text{elsewhere.}
\end{array}\right.
\end{equation} 

\begin{figure}
  \begin{center}
    \includegraphics[width=0.45\textwidth]{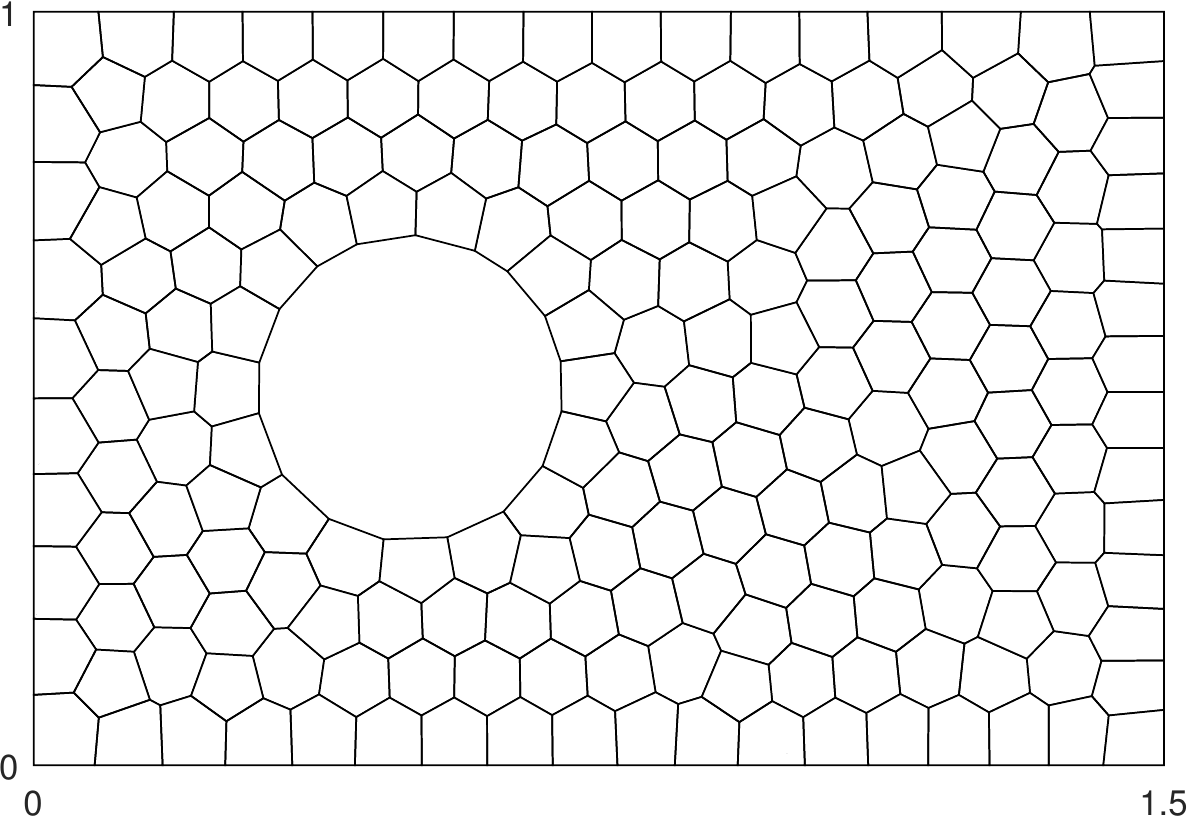}
    \includegraphics[width=0.45\textwidth]{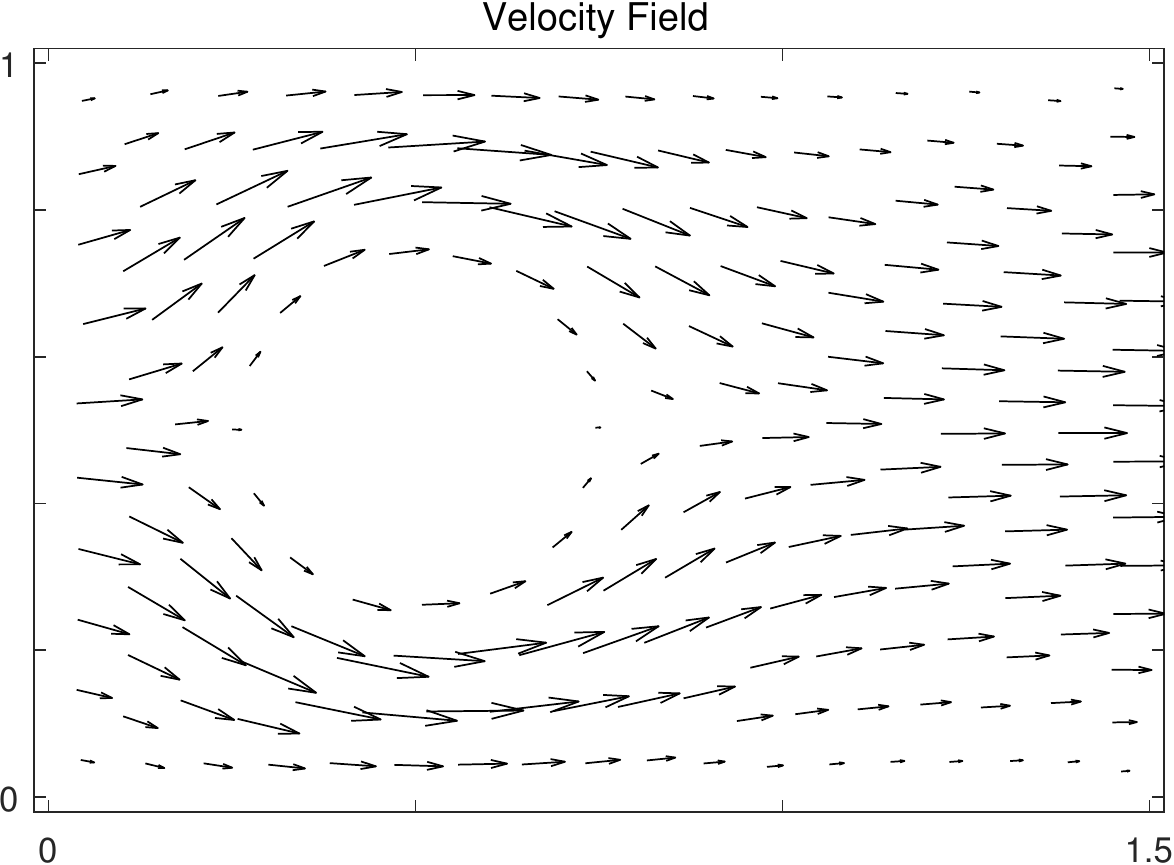}
    \caption{The velocity and scaled velocity field of flow around
      cylinder}
  \label{flow_around_cylinder_fig}
  \end{center}
\end{figure}
Figure \ref{flow_around_cylinder_fig} shows the polygonal mesh and the
corresponding velocity field. The proposed method provides a smooth
solution which is agreed with the expected behavior. and it can handle
a wide variety of polygonal elements without additional techniques.

\section{Conclusions}
The discontinuous Galerkin method by locally divergence-free patch
reconstruction for simulation of incompressible Stokes problems is
developed. The interior penalty method with solenoidal velocity field
allows calculating the velocity field with no presence of
pressure. This method can achieve high order accuracy with one degree
of freedom per element. In spite of such advantage, our method suffers
from the inconsistence penalty term $h^{-(m+1)} $ which leads to an
ill-posed linear system for large $m$.


\bibliographystyle{abbrv}

\end{document}